\documentclass[12pt]{article}

\usepackage[letterpaper,textwidth=13.5cm,centering]{geometry}
\usepackage{amssymb,latexsym,amsmath,amscd,array,graphicx,amsthm}
\usepackage{mathtools}
\usepackage{tabularx}
\usepackage{array}
\usepackage{booktabs}
\usepackage{float}
\usepackage{placeins}
\usepackage[hidelinks]{hyperref}
\usepackage{caption}

\numberwithin{equation}{section}

\newcommand{\R}{\mathbb{R}}

\newcommand{\Z}{\mathbb{Z}}
\newcommand{\N}{\mathbb{N}}
\newcommand{\id}{\mathrm{id}}

\newcommand{\sgn}{\mathrm{sgn}}

\newcommand{\len}{\mathrm{len}}
\newcommand{\diam}{\mathrm{diam}}
\newcommand{\dist}{\mathrm{dist}}
\newcommand{\conv}{\mathrm{conv}}
\newcommand{\circumference}{\mathrm{circumference}}
\newcommand{\link}{\mathrm{link}}

\theoremstyle{plain}
\newtheorem{thm}{Theorem}[section]
\newtheorem{cor}[thm]{Corollary}
\newtheorem{lem}[thm]{Lemma}
\newtheorem{prop}[thm]{Proposition}

\newtheorem{result}[thm]{Result}

\theoremstyle{definition}
\newtheorem{defin}[thm]{Definition}

\newtheorem{rem}[thm]{Remark}
\newtheorem{example}[thm]{Example}

\newcommand\theoref{Theorem~\ref}
\newcommand\lemref{Lemma~\ref}
\newcommand\remref{Remark~\ref}
\newcommand\propref{Proposition~\ref}

\newcommand\defref{Definition~\ref}

\newcommand\sectref{Section~\ref}
\newcommand\subsectref{Subsection~\ref}
\newcommand\figref{Figure~\ref}

\makeatletter
\newenvironment{proofsketch}[1][\proofname]{%
  \par\pushQED{\qed}%
  \normalfont\topsep6\p@\@plus6\p@\relax
  \trivlist\item[\hskip\labelsep
    \itshape
    Sketch of #1.]\ignorespaces
}{%
  \popQED\endtrivlist\@endpefalse
}
\makeatother

\title{Gordian split links in the\\Gehring ropelength problem}
\author{Friedrich Bauermeister\\[1ex]
\small Department of Mathematics, Dartmouth College\\
\small \texttt{friedrich.bauermeister.gr@dartmouth.edu}}
\date{}

\begin{document}

\maketitle

\begin{abstract}
A thick link is a link in $\mathbb{R}^3$ such that each component of the link lies at distance at least $1$ from every other component. Strengthening the notion of thickness, a thickly embedded link is a thick link whose open radius-$\tfrac{1}{2}$ normal disk bundles of all components are embedded. A thick homotopy is a link homotopy of a thick link that preserves thickness and total length throughout. A thick isotopy is a link isotopy of a thickly embedded link that preserves thick-embeddedness and total length throughout. We construct an isotopically Gordian split link, that is, a thickly embedded 4-component link which is topologically split but which cannot be split by a thick isotopy. This is the first time a Gordian split link is shown to exist in this most permissive setting where length trading between components is allowed. We then prove for the first time that local, non-global minima for Gehring ropelength exist. In particular, we construct a 2-component homotopically Gordian unlink, that is, a link in the link homotopy class of the unlink which cannot be split by any thick homotopy.
\end{abstract}

\vspace{.4\baselineskip}
\begingroup
\par
\endgroup

\section{Introduction}

In 2015, Coward and Hass \cite{TopAndPhysLinkTheoryDist} defined the notion of a physical link isotopy. A physical link isotopy is a link isotopy that keeps a link thickly embedded throughout, as well as keeping the length of each component constant. This made it possible to study a physical knot- and link theory that corresponds to knots and links tied from real pieces of rope. In the same paper, Coward and Hass construct a $2$-component \textit{Gordian split link} which they define as a thickly embedded, topologically split link that cannot be split by a physical link isotopy. More generally, a \textit{Gordian pair of links} is a set of two links in $\R^3$ which are topologically equivalent, but which are not equivalent under physical deformations. Here \textit{topologically equivalent} refers either to being in the same link isotopy or homotopy class, and \textit{physical deformation} refers to either link homotopy or link isotopy, subject to additional geometric constraints on length and thickness. For length, there are two common constraints: either demanding that the length of each component individually does not increase throughout the homotopy/isotopy, or that their sum does not increase throughout the homotopy/isotopy. If only the sum of lengths is constrained, we say that \textit{length trading} between components is permitted. The thickness constraint usually demands that links stay thick links (see \defref{definition: thick link}) or thickly embedded links (see \defref{definition: thickly embedded link}) throughout the homotopy or isotopy respectively. When considering link isotopy, sometimes additional curvature constraints are imposed.\\

The original Gehring ropelength problem asked for the minimal length of a closed curve in $\R^3$, given that it is linked with another curve from which it keeps at least unit distance. The answer was quickly shown to be $2\pi$, realized by two Hopf-linked unit circles. Given some link homotopy class, the general Gehring ropelength problem \cite{CriticalityForGehringProblem} asks for the minimal ropelength of a thick link in that homotopy class. A link realizing this minimal length is called an \textit{ideal link} of the homotopy class. While the precise length of ideal links is known only in a few cases, there are various upper and lower bounds. Cantarella, Fu, Kusner, Sullivan, and Wrinkle \cite{CriticalityForGehringProblem} defined a notion of criticality for the general Gehring ropelength problem, which is a necessary but not sufficient criterion for being an ideal link.\\

The paper is organized as follows. In \sectref{Main results}, we state the main results, and in \sectref{Overview of existing results about Gordian pairs}, we place them in the context of previously known examples of Gordian pairs. \sectref{Preliminaries} introduces the relevant notions of thickness, Gordian pairs, and minimizers of ropelength. In \sectref{bounds on ropelength}, we collect the lower bounds on ropelength used in the proofs. \sectref{An isotopically Gordian split link exists} constructs a $4$-component Gordian split link. Finally, \sectref{Homotopically Gordian pairs exist in every link homotopy class with 2 components} constructs local, non-global minima for ropelength and Gordian pairs in every link homotopy class with two components.

\section{Main results}\label{Main results}
We briefly state the main results of this paper. All notions of thick links, thickly embedded links, thick link homotopies, thick link isotopies, and the notion of being a sink for ropelength referred to in this main results section are as in Definitions~\ref{definition: thick link}, \ref{definition: thickly embedded link}, \ref{definition: thick homotopy/isotopy}, and \ref{definition: sink of ropelength}. 

\begin{result}[\theoref{theorem: A Gordian split link exists}]
There exists a $4$-component thickly embedded link that is topologically split but which cannot be split by a thick homotopy. 
\end{result}
This proves the existence of a configuration that is physical in the strongest sense which cannot be split even under the most permissive of thickness constraints. Thus, the configuration is a Gordian split link both as a thick link and as a thickly embedded link. 

\begin{result}[\theoref{theorem: Gordian pairs exist in all link homotopy classes of two components}]
For every link homotopy class with $2$ components, there exists a thick link in that class which is a non-global, local minimum and a sink for ropelength. In particular, there are Gordian pairs in every link homotopy class with $2$ components.
\end{result}
This is the first time local, non-global minimizers in the Gehring ropelength problem have been proven to exist at all, and the proof is constructive. In particular, this settles the question of whether homotopically Gordian unlinks exist in the affirmative.
\begin{result}[\theoref{theorem: a Gordian unlink exists}]
There exists a $2$-component thick link in the link homotopy class of the unlink which cannot be split by a thick link homotopy.
\end{result}
Together, these results demonstrate that thickness constraints meaningfully obstruct the splitting of links, and that the energy landscape of Gehring ropelength has many local minima.  

\section{Overview of existing results\\
about Gordian pairs}\label{Overview of existing results about Gordian pairs}

\FloatBarrier
\begin{table}[H]
  \small
  \centering
  \renewcommand{\arraystretch}{1.3}

  \begin{tabularx}{\textwidth}{|>{\raggedright\arraybackslash}X 
                               |>{\raggedright\arraybackslash}X 
                               |>{\raggedright\arraybackslash}X 
                               |>{\raggedright\arraybackslash}X 
                               |>{\raggedright\arraybackslash}X|}
  \hline
  \multicolumn{5}{|c|}{\textbf{Isotopically Gordian Pairs}} \\
  \hline
   & Any Gordian pair & Split link & $2$-component split link & $2$-component unlink \\
  \hline
  Without length trading and with extra curvature constraints 
  & Coward and Hass \cite{TopAndPhysLinkTheoryDist} 
  & Coward and Hass \cite{TopAndPhysLinkTheoryDist} 
  & Coward and Hass \cite{TopAndPhysLinkTheoryDist} 
  & Ayala \cite{ThinGordianUnlinks} \\
  \hline
  Without length trading 
  & Coward and Hass \cite{TopAndPhysLinkTheoryDist} 
  & Coward and Hass \cite{TopAndPhysLinkTheoryDist} 
  & Coward and Hass \cite{TopAndPhysLinkTheoryDist} 
  & Ayala and Hass \cite{GordianUnlinks},  Kusner and Kusner \cite{XaraxUnlinks} \\
  \hline
  With length trading
  & Kusner and Kusner \cite{GordianPairOfLinks} 
  &  \sectref{An isotopically Gordian split link exists}
  &  
  & \\
  \hline
  With length trading and allowing homotopies instead of isotopies 
  & Kusner and Kusner \cite{GordianPairOfLinks} 
  & \sectref{An isotopically Gordian split link exists}
  &  
  & \\
  \hline
  \end{tabularx}

  \vspace{1em}

  \begin{tabularx}{\textwidth}{|>{\raggedright\arraybackslash}X 
                               |>{\raggedright\arraybackslash}X 
                               |>{\raggedright\arraybackslash}X 
                               |>{\raggedright\arraybackslash}X 
                               |>{\raggedright\arraybackslash}X|}
  \hline
  \multicolumn{5}{|c|}{\textbf{Homotopically Gordian Pairs}} \\
  \hline
   & Any Gordian pair & Split link & $2$-component unlink & All $2$-component links \\
  \hline
  Without length trading 
  & Kusner and Kusner \cite{GordianPairOfLinks} 
  & \sectref{An isotopically Gordian split link exists}, \sectref{Homotopically Gordian pairs exist in every link homotopy class with 2 components}
  & \sectref{Homotopically Gordian pairs exist in every link homotopy class with 2 components}
  & \sectref{Homotopically Gordian pairs exist in every link homotopy class with 2 components} \\
  \hline
  With length trading 
  & Kusner and Kusner \cite{GordianPairOfLinks} 
  & \sectref{An isotopically Gordian split link exists}, \sectref{Homotopically Gordian pairs exist in every link homotopy class with 2 components}
  & \sectref{Homotopically Gordian pairs exist in every link homotopy class with 2 components}
  & \sectref{Homotopically Gordian pairs exist in every link homotopy class with 2 components} \\
  \hline
  \end{tabularx}

  \caption{Summary of known examples for Gordian pairs in different settings. }
  \label{tab:GordianPairs}
\end{table}
\FloatBarrier

The first Gordian split link was constructed by Coward and Hass in 2015 \cite{TopAndPhysLinkTheoryDist}. In 2023, Kusner and Kusner \cite{GordianPairOfLinks} constructed a Gordian pair of thickly embedded $7$-component links. The links in the pair are link isotopic to each other, but there is no thick homotopy between them. So even though the links are equivalent in the strongest topological sense, they are not equivalent even in the weakest physical sense. It was Kusner and Kusner who first coined the term \textit{Gordian pair}. In \cite{ThinGordianUnlinks} Ayala constructs a thickly embedded $2$-component link in the isotopy class of the unlink which cannot be split by thick isotopy under additional curvature constraints. This was the first example of a Gordian unlink. Recently, Ayala and Hass constructed a different example of a Gordian unlink without needing any additional curvature constraints. Independently, Kusner and Kusner announced the construction of a Gordian unlink \cite{XaraxUnlinks}. Table \ref{tab:GordianPairs} summarizes both the existing results on Gordian pairs and the author's contributions in this paper. The table is split into two parts, considering links with respect to isotopy and homotopy respectively. The columns of each table denote what kind of example was constructed. The rows of each table denote which geometric constraints were considered, with the strength of constraints being relaxed from top to bottom. The paper referenced in each cell is the paper in which the \textit{first} example of the corresponding Gordian pair was constructed. The tables give partial orders: the further down and to the right a cell in a given table is, the harder it is to construct an example for it.

\section{Preliminaries}\label{Preliminaries}

\begin{defin}
Let $A\subseteq\R$ be an interval with non-empty interior, or let $A=S^1$. A \textit{curve} is a continuous map $C:A\to\R^3$. If $A\subseteq\R$, we define
$$\len(C)\coloneqq\sup\left\{\sum_{i=1}^n\left\|C(t_i)-C(t_{i-1})\right\|\ \middle|\ n\geq 1,\ t_0<t_1<\dots<t_n,\ t_i\in A\right\}.$$
If $A=S^1$, we define $\len(C)\coloneqq\len(C\circ q)$, where $q:[0,1]\to S^1$ is the standard quotient map. We call $C$ \textit{rectifiable} if $\len(C)<\infty$. If $A=S^1$, or if $A$ is compact and $C(\min A)=C(\max A)$, we call $C$ a \textit{closed curve}. Unless otherwise stated, no regularity beyond continuity is assumed.
\end{defin}

\begin{defin}
Let $n\in\N$ and let $S_1,S_2,\dots,S_n$ be copies of $S^1$. By an $n$-component \textit{link} we mean a continuous map
$$L=(C_1,\dots,C_n):S_1\sqcup S_2\sqcup\dots\sqcup S_n\to\R^3$$
such that $C_i(S_i)\cap C_j(S_j)=\emptyset$ whenever $i\neq j$. The individual components are not required to be injective. In the literature such maps are commonly called link maps, but throughout this paper we will refer to them as links.
\end{defin}

\begin{defin}
Let
$$H:[0,1]\times(S_1\sqcup S_2\sqcup\dots\sqcup S_n)\to\R^3$$
be continuous. We call $H$ a \textit{link homotopy} if $H(t,\cdot)$ is a link for every $t\in[0,1]$. We call $H$ a \textit{link isotopy} if $H(t,\cdot)$ is an embedding for every $t\in[0,1]$.
\end{defin}

\begin{defin}
Let $A,B \subseteq \R^n $ be non-empty sets, let $x\in \R^n$ and let $f,g$ be functions with image in $\R^n$. Then we denote the distance of $A$ to $B$ with $\dist(A,B)\coloneqq \inf\{||a-b||\mid a\in A, b\in B\}$; and we define $\dist(f,g)$, $\dist(f,A)$, $\dist(x,A)$, $\dist(x,f)$ by identifying $x$ with $\{x\}$ and $f,g$ with their image. If $f,g:M\to \R^n$ then $||f-g||\coloneqq \sup\{||f(x)-g(x)||\mid x \in M\}$.
\end{defin}
\begin{defin}
If $L = \bigcup_{i=1}^n C_i$ is a link, then we define its \textit{ropelength} as the sum of the lengths of its components.
$$\len(L) = \sum_{i=1}^n \len(C_i).$$
\end{defin}

\begin{defin}\label{definition: thick link}
Let $L = \left(C_1,\dots,C_n\right) : S_1 \sqcup \dots \sqcup S_n \to \R^3$ be a link. We say that $L$ has \textit{thickness} $d$ if
$$\dist(C_i,C_j) \geq d$$
for every $i,j \in \{1,\dots,n\}, i\neq j$.  We define $\tau(L)$ as the largest such $d$. If $\tau(L) \geq 1$ we call $L$ a \textit{thick link}.
\end{defin}

\begin{defin}\label{definition: thickly embedded link}
We say that $L$ is \textit{$d$-thickly embedded} if $L$ is $\mathcal{C}^1$ and if the open radius-$\tfrac{d}{2}$ normal disk bundle around $L$ is embedded. We define $\tau_e(L)$ as the largest such $d$. If $\tau_e(L)\geq 1$ we call $L$ a \textit{thickly embedded link}.
\end{defin}

\begin{defin}\label{definition: thick homotopy/isotopy}
Let $H:[0,1] \times \left(S_1 \sqcup S_2 \sqcup \dots \sqcup S_n\right) \to \R^3$ be a link homotopy (isotopy). We call $H$ a \textit{$d$-thick homotopy (isotopy)} if for all $t\in [0,1]$:
\begin{itemize}
\item $L_t = H(t,\cdot)$ has thickness $d$ (is $d$-thickly embedded),
\item $\len(L_t)\leq \len(L_0)$. 
\end{itemize}
If $d=1$ we call $H$ a \textit{thick homotopy (isotopy).}
\end{defin}

\begin{defin}\label{definition: reachable}
For thick links $L_1$ and $L_2$ we define
\begin{align*}
L_1 \preceq L_2 &\iff \text{ There is a thick link homotopy from }L_2 \text{ to } L_1.\\
L_1 \simeq L_2 &\iff L_1 \preceq L_2 \text{ and } L_2 \preceq L_1
\end{align*}
For thickly embedded links $L_1$ and $L_2$ we define
\begin{align*}
L_1 \preceq_e L_2 &\iff \text{ There is a thick link isotopy from }L_2 \text{ to } L_1.\\
L_1 \simeq_e L_2 &\iff L_1 \preceq_e L_2 \text{ and } L_2 \preceq_e L_1.
\end{align*}
If $L_1 \preceq L_2$ or $L_1 \preceq_e L_2$ we say that $L_1$ is \textit{reachable} from $L_2$.
\end{defin}

\begin{prop}
If $L_1 \preceq L_2$, then
$$\len(L_1) = \len(L_2) \iff L_2 \preceq L_1 \iff L_1 \simeq L_2,$$
and the same for $\preceq_e$.
\end{prop}

\begin{proof}
Only the implication $\len(L_1) = \len(L_2) \implies L_2 \preceq L_1$ is interesting. Let $H_t$ be the thick homotopy (isotopy) from $L_2$ to $L_1$. Since $\len(L_1)=\len(L_2)$, the reverse homotopy $H_{1-t}$ is a thick homotopy (isotopy) from $L_1$ to $L_2$.
\end{proof}

\begin{defin}\label{definition: Gordian pair of links}
A pair of thick (thickly embedded) links $(L_1,L_2)$ is called \textit{Gordian as thick links (as thickly embedded links)} if $L_1$ and $L_2$ are in the same link homotopy (isotopy) class, but neither is reachable from the other. 
\end{defin}
We will refer to a pair that is Gordian as thick links as \textit{homotopically Gordian}, and to a pair that is Gordian as thickly embedded links as \textit{isotopically Gordian}. When it is clear from context whether we are working in the setting of thick links or thickly embedded links, we will simply call such a pair of links a \textit{Gordian pair}.

For $S\subseteq\R^n$, we denote its convex hull by $\conv(S)$.

\begin{defin}
A thick (thickly embedded) link $L = \bigcup_{i\in I} C_i$ is called \textit{separated} if there is some $\emptyset \subsetneq J \subsetneq I$ such that 
$$\conv\left(\bigcup_{i\in J} C_i\right) \cap \conv\left(\bigcup_{i\in I\setminus J} C_i\right) = \emptyset.$$
A thick (thickly embedded) link $L$ is called \textit{split as a thick (thickly embedded) link} if there is a thick (thickly embedded) link homotopy (isotopy) from $L$ to some separated link $L'$.
\end{defin}

\begin{defin}\label{definition: Gordian split link}
A thick (thickly embedded) link $L$ is called a homotopically (isotopically) \textit{Gordian split link} if $L$ is topologically split through a link homotopy (isotopy), but not split as a thick (thickly embedded) link.
\end{defin}
It is easy to see that every Gordian split link $L$ is part of a Gordian pair. If we can write $L=L_1 \cup L_2$ where $L_1$ and $L_2$ are not topologically linked with each other, define $L' = L_1 \cup T(L_2)$ for some translation $T$ sufficiently far away from $L_1$. Clearly $L$ and $L'$ are in the same link homotopy (isotopy) class. But since $L$ was a Gordian split link, we have $L'\npreceq L$. Since $\len(L)=\len(L')$, we have $L\npreceq L'$. Hence $(L,L')$ is a Gordian pair.

\begin{rem}
For the link homotopy class with one component, we have $$L_1 \preceq L_2 \iff \len(L_1) \leq \len(L_2).$$
Therefore, no Gordian pair can exist in the $1$-component link homotopy class. 
\end{rem}

\begin{rem}
It is not known whether there is a Gordian pair of thickly embedded knots. There is some computational evidence for the existence of a Gordian unknot, i.e., a thickly embedded knot in the isotopy class of the unknot which has no thick isotopy to a circle \cite{GordianUnknot}. 
\end{rem}

\begin{defin}
We call a thick link $L$ a \textit{global minimum} for ropelength if for every thick link $L'$ in the same link homotopy class as $L$, we have
$$\len(L')\geq \len(L).$$
We call a thick link $L$ a \textit{local minimum} for ropelength if there is some $\varepsilon>0$ such that for every thick link $L'$ 
$$||L-L'||<\varepsilon \implies \len(L')\geq \len(L).$$
\end{defin}

Global and local minima of ropelength are interesting in and of themselves, and much is still unknown about them. For example, up until the writing of this article, the only local minima proven to exist were also global minima. In \cite{CriticalityForGehringProblem}, Cantarella, Fu, Kusner, Sullivan, and Wrinkle use their criterion of ropelength criticality to construct \textit{suspected} local and global minimizers in different link homotopy classes, for example for the Borromean rings. When studying thick homotopies, there is another interesting kind of minimization.

\begin{defin}[sink of ropelength]\label{definition: sink of ropelength}
If a thick link $L$ is a minimal element of the preorder $\preceq$, then we call $L$ a \textit{sink} of the ropelength functional.
\end{defin}
Sinks are exactly the configurations which cannot be shortened by a thick homotopy. We briefly define the notion of a sink more generally, and see how it relates to the notion of local minima.
\begin{defin}
Let $X$ be a topological space, $f:X\to \R$ a (not necessarily continuous) function. We call $x\in X$ a \textit{sink} of $f$ if for all continuous $\gamma:[0,1]\to X$ with $\gamma(0)=x$, we have
$$(f\circ\gamma)(t) \leq f(x) \text{ for all } t \implies f\circ \gamma \text{ constant.}$$
\end{defin}

It is easy to see that every global minimum and every strict local minimum is a sink. However, even for nice spaces $X$ and continuous functions $f$, there are local minima which are not sinks and sinks which are not local minima.

\begin{example}
Take $X=\R$ and $f(x)=\min(x,0)$. Then $x=1$ is a local minimum, but not a sink.
\end{example}

\begin{example}
Take $X=\R$ and $f(x)=\begin{cases}x\sin\left(\frac{1}{x}\right), &x\neq 0\\ 0, &x=0\end{cases}$\\ 
Then $x=0$ is a sink, but not a local minimum.
\end{example}

\begin{defin}[strictly locally minimizing set]
Let $X$ be a topological space and $f:X\to \R$. We say that a subset $S\subseteq X$ is a \textit{strictly locally minimizing set} of $f$ if $f$ is constant on $S$ and if for all $x\in S$, there is an open neighborhood $U \subseteq X$ of $x$ such that
$$y\in U \implies f(y)\geq f(x)$$
with equality only if $y \in S$.
\end{defin}

\begin{prop}\label{proposition: every point of a closed, strictly minimizing set is a sink}
Let $X$ be a topological space, $f:X\to \R$, and let $S$ be a strictly locally minimizing set of $f$. If $S\subseteq X$ is closed, then every $x\in S$ is a sink of $f$. 
\end{prop}

\begin{proof}
Let $c=f(x)$, let $\gamma:[0,1]\to X$ with $\gamma(0)=x$ and with $f(\gamma(t))\leq c$ for all $t\in [0,1]$. Let $J =\gamma^{-1}(S)$. Clearly $0\in J$, so $J\neq \emptyset$. Further, $J$ is the preimage of a closed set under a continuous map, hence closed. Next we show that $J$ is open. Let $t\in J$. Since $\gamma(t)\in S$ we have $f(\gamma(t))=c$. Further there is a neighborhood $U_t$ around $t$ with $$t'\in U_t \implies  c=f(\gamma(t)) \leq f(\gamma(t'))$$
and with equality if and only if $\gamma(t')\in S$. By assumption $f(\gamma(t'))\leq c$. Therefore $f(\gamma(t'))=c$ and hence $\gamma(t') \in S$ . This shows that $J$ is open. Since $J$ is a non-empty, closed and open subset of $[0,1]$ we have $J=[0,1]$. So $\gamma([0,1])\subseteq S \subseteq f^{-1}(c)$, hence $f\circ \gamma$ is constant.
\end{proof}

\section{Bounds on ropelength}\label{bounds on ropelength}

\begin{thm}[Original Gehring Problem]\label{theorem: Gehring Problem}
Let $C_1$ and $C_2$ be two linked closed curves in $\R^3$ with $\dist(C_1,C_2)\geq 1$. Then $\len(C_1)\geq 2\pi$. If $\len(C_1)=2\pi$, then $C_1$ parametrizes a unit circle.
\end{thm}
An elegant proof of this theorem was found by Marvin Ortel, and was first published in \cite{CriticalityForGehringProblem}. The statement can be strengthened and generalized in various ways. We will give a few of the most useful generalizations and some important consequences. 

\begin{defin}
Let $C:S^1\to \R^3$ be a closed curve, and let $p \in \R^3$. We define the \textit{$p$-cone} of $C$ as the disk $K:\left(S^1\times [0,1]\right)/_{\sim} \to \R^3$, $K(x,y)=y\cdot C(x) + (1-y)\cdot p$. Here $(x,y)\sim(x',y') \iff y=0=y'$. If $p$ is the center of mass of $C$, we call $K$ the \textit{center-of-mass spanning disk} of $C$.
\end{defin}

\begin{thm}\label{theorem: 2Pi r + circumference}
Let $C$ be a closed curve, $p$ a point in the convex hull of $C$, $K$ the $p$-cone of $C$. Let $r\geq 0$ and let $x_1,x_2,\dots,x_n \in K$ be the vertices of a convex polygon $P \subseteq K$, such that $ \dist(x_i,C)\geq r$ for all $i=1,\dots,n$. Then
$$\len(C) \geq 2r\pi  + \circumference(P).$$
\end{thm}
\begin{proof}
If $C$ is not rectifiable, the statement is trivial. Suppose therefore that $C$ is rectifiable. The argument of \cite[Lemma 2.2]{GordianUnlinks} shows that the cone angle of $K$ at $p$ is at least $2\pi$. Although the result is stated there for center-of-mass spanning disks, its proof only uses that $p\in\conv(C)$. Extending the rays of $K$ beyond $C$ therefore gives a complete $CAT(0)$ cone surface. Coward and Hass establish the corresponding boundary-length estimate on such surfaces in \cite[Proposition 2.1 and the subsequent remark]{TopAndPhysLinkTheoryDist}. Replacing their triangle by $P$ and projecting radially onto circles of radius $r$, their argument gives $$\len(C)\geq 2\pi r+\circumference(P).$$
\end{proof}
For $n=2$, we get the following version of this theorem.
\begin{prop}\label{proposition: two-point arbitrary radius inequality}
Let $C$ be a closed curve, $p \in \conv(C)$, $K$ the $p$-cone of $C$. Let $r\geq 0$ and let $x,y \in K$ be such that $\dist(x,C),\dist(y,C) \geq r$. Then $$\len(C) \geq 2\pi r + 2||x-y||.$$
\end{prop}
We will use this result for $r=1$ in \propref{proposition: main inequality without epsilons} and for $r=1-\varepsilon$ in \propref{proposition: main inequality with epsilons.}, for small $\varepsilon$. 

\begin{prop}\label{proposition: main inequality without epsilons}
Let $C$ be a closed curve, $p$ a point in the convex hull of $C$, and $K$ the $p$-cone of $C$. Let $A\subseteq K$ be such that $ \dist(A,C)\geq 1$. Then $$\len(C) \geq 2\pi  + 2\diam(A).$$

\end{prop}

We can now use these results to give some lower bounds for the lengths of curves in a thick link that are linked with a certain number of other curves. We already know from the original Gehring problem that if $(C_1,C_2)$ is a thick link with $C_1$ and $C_2$ topologically linked, then $\len(C_1)\geq 2\pi$. It is natural to ask about the general case: If $(C_1,C_2,\dots,C_{n+1})$ is a thick link with $C_1$ topologically linked with each of $C_2,\dots,C_{n+1}$, what is the minimal length that $C_1$ can have? This question was answered by Cantarella, Kusner, and Sullivan in \cite{MinimumRopelengthOfKnotsAndLinks}. Their formulation was originally only for the case of thickly embedded links, but it holds just as well for thick links. 
\begin{thm}\label{theorem: lower bound for ropelength of thick links}
Let $(C_1,C_2,\dots,C_{n+1})$ be a thick link such that $C_1$ is topologically linked with each of $C_2,\dots,C_{n+1}$. Then
$\len(C_1) \geq 2\pi + Q_n$, where $Q_n$ is the length of the shortest closed curve in $\R^2$ surrounding $n$ points of at least unit distance to each other. Here $Q_1=0$, $Q_2 =2$, $Q_3=3$, $Q_4=4$. Asymptotically, $Q_n \sim \sqrt{n}$ \cite{MinimumRopelengthOfKnotsAndLinks}.
\end{thm}

If $L=(C_1,\dots,C_n)$ is a thick link such that each curve $C_i$ of $L$ is topologically linked with $k_i$ curves, then $\len(L) \geq 2n\pi + \sum_{i=1}^n Q_{k_i}$.  If $\len(L)=2n\pi + \sum_{i=1}^n Q_{k_i}$, then $L$ is a global minimum for ropelength in its link homotopy class. To the author's knowledge, these are the only configurations that have been proved to be global minimizers. Some examples and graphics of such global minimizers can be found in \cite{CriticalityForGehringProblem}.

\begin{prop}\label{proposition: main inequality with epsilons.}
Let $C$ be a closed curve, $p$ a point in the convex hull of $C$, $K$ the $p$-cone of $C$. Let $0\leq \varepsilon \leq 1$ and let $x,y \in \R^3$ be such that $\dist(x,C),\dist(y,C) \geq 1$ and $\dist(x,K),\dist(y,K)\leq \varepsilon$. Then $$\len(C) \geq 2\pi  + 2||x-y|| - (2\pi+4)\varepsilon.$$
\end{prop}
\begin{proof}
Let $x',y'\in K$ with $||x'-x||,||y'-y||\leq \varepsilon$. Then $\dist(x',C),\dist(y',C)\geq 1-\varepsilon.$ Hence by \propref{proposition: two-point arbitrary radius inequality}
\begin{align*}
\len(C) \geq &2\pi(1-\varepsilon) + 2||x'-y'||\\
\geq &2\pi(1-\varepsilon) + 2||x-y||-4\varepsilon\\
= &2\pi + 2||x-y|| -(2\pi+4)\varepsilon.
\end{align*}
\end{proof}
There also exist various ways to bounds the length of thickly embedded knots based on their topology, see for example \cite{QuadrisecantsRopelength,LowerBoundsThickKnots,RopelengthAlmostLinear}.

\section{An isotopically Gordian split link exists}\label{An isotopically Gordian split link exists}
The construction of our Gordian split link, as well as the basic structure of the proof, has a very similar flavor to that of Coward and Hass \cite{TopAndPhysLinkTheoryDist}. Both approaches derive a lower bound for the length needed to split a thickly embedded link by assuming that there is a thick homotopy/physical isotopy, tracking the intersections of a center-of-mass spanning disk and a curve throughout that homotopy/isotopy, and finally arguing that there needs to be some point in time where the intersections yield some lower bound through applications of propositions from \sectref{bounds on ropelength}. The main difference is in the book-keeping of how intersections are tracked. Coward and Hass need to resolve multiple kinds of singularities that might occur during the tracking of intersections. Our arguments use the homotopy invariance of the Brouwer degree and the mountain climbing theorem instead. We start out by stating the main proposition of this chapter, which establishes a lower bound for the length necessary to split some thick links.

\subsection{A lower bound on ropelength needed to split some thick links}\label{subsection: A lower bound on ropelength needed to split some thick links}

\begin{prop}\label{proposition: lower bound on splitting four component link}
Consider a thick link $L$ consisting of four curves: a main curve $M$ linked with two auxiliary curves $A^1$ and $A^2$, and a fourth curve $C$ which we will call the center curve. Let $D$ be the center-of-mass spanning disk of $C$. Assume the following:
\begin{itemize}
  \item  $D$ and $M$ are $\mathcal{C}^1$ and they intersect in exactly two points, i.e.\ there are exactly two pairs of elements $(a,v),(b,w) \in S^1 \times D^1$ such that $M(a)=D(v), M(b)=D(w)$,
  \item these intersections are transverse,
  \item and $a\neq b$.
\end{itemize}
We get that $a$ and $b$ partition $S^1$ into two parts, say $I^1$ and $I^2$. Let $M^1$ be the concatenation of $M|_{I^1}$ with the straight line segment connecting $M(a)$ and $M(b)$ and let $M^2$ be the concatenation of $M|_{I^2}$ and the straight line segment connecting $M(a)$ and $M(b)$. Further assume: 
\begin{itemize}
\item $M^1$ is linked with $A^1$,
\item $M^2$ is linked with $A^2$.
\end{itemize}
Then, if $H_t$ is a homotopy of $L$ preserving thickness throughout, and such that $\conv(C_1)$ and $\conv(M_1)$  are disjoint, we have $\sup_{t\in[0,1]} \left(\len(C_t)\right) \geq 2\pi +4$.
\end{prop}

Consequently,
\begin{align*}\sup_{t\in[0,1]}\len(H_t) = &\sup_{t\in [0,1]}\left(\len(M_t) + \len(A^1_t)+\len(A^2_t)+\len(C_t)\right)\\
\geq &\left(2\pi+2\right)+2\pi+2\pi+\sup_{t\in [0,1]}\len(C_t)\\
\geq &8\pi+6,
\end{align*}
where we used that $\len(M_t)\geq 2\pi+2$ and $\len(A^1_t),\len(A^2_t)\geq 2\pi$ by \theoref{theorem: Gehring Problem}  and \theoref{theorem: lower bound for ropelength of thick links}. In particular, if $\len(L) < 8\pi+6$, there can be no thick homotopy splitting $L$.

\begin{figure}[H]
        \centering
        \includegraphics[width=0.5\textwidth]{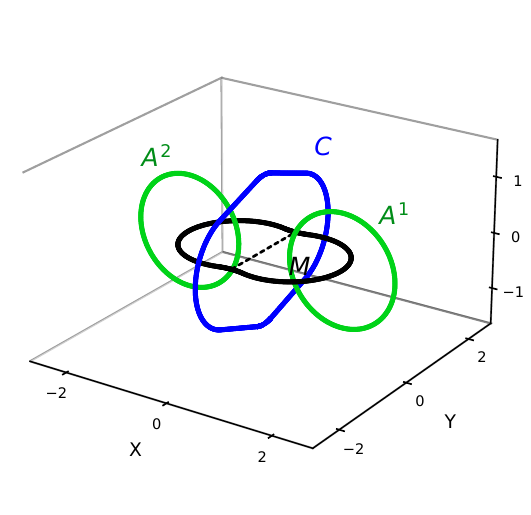} 
        \caption{This plot shows the main curve $M$ in black, the two auxiliary curves $A^1$ and $A^2$ in green and the center curve $C$ in blue. The dotted line joins $M(a)$ and $M(b)$.}
        \label{figure: 4 component Gordian split link with dashed line}
\end{figure}

\begin{proofsketch}[proof of \propref{proposition: lower bound on splitting four component link}]
Assume we have some homotopy $H_t$, $t\in[0,1]$, preserving thickness throughout, which splits the link in \figref{figure: 4 component Gordian split link with dashed line}. As the center curve $C_t$ varies continuously, so does the center-of-mass spanning disk $D_t$ of $C_t$. Imagine there are continuous functions $t\mapsto a_t,b_t\in S^1$ such that $a_0=a$, $b_0=b$, such that $M_t(a_t), M_t(b_t) \in D_t$ and such that for some $T\in [0,1]$ we have $a_T=b_T$. Then $a_t$ and $b_t$ split $S^1$ into intervals $I^1_t$ and $I^2_t$. We can define continuously varying families of curves $M^1_t$ and $M^2_t$ as concatenations of $M_t|_{I^1_t}$ and $M_t|_{I^2_t}$ with the straight line segment from $M_t(a_t)$ to $M_t(b_t)$ respectively. Either $M^1_T$ or $M^2_T$ becomes a single point, while the other one becomes all of $M_T$. Without loss of generality assume $M^1_T$ becomes a single point. Then $(M^1_T,A^1_T)$ becomes unlinked. But $(M^1,A^1)$ was linked, so at some point in time $t\in[0,T]$, $M_t^1$ and $A_t^1$ have to intersect. This intersection cannot be an intersection between $M_t$ and $A^1_t$, so it has to be an intersection of the straight line segment with $A^1_t$. Since the endpoints of that straight line segment are on $M_t$, hence have distance at least $1$ from $A^1_t$, we can conclude that the straight line segment at time $t$ has length at least $2$. This means that $M_t(a_t)$ and $M_t(b_t)$ are at distance at least $2$ from each other, and since they are both in the center-of-mass spanning disk of $C_t$, we get $\len(C_t) \geq 2\pi + 4$ by \propref{proposition: main inequality without epsilons}.
\end{proofsketch}

The main step omitted in this sketch is how, given a homotopy $H$ preserving thickness, we can continuously keep track of the intersection points $a_t$ and $b_t$, and why we can assume that they merge eventually. In this form the claim is wrong. Indeed, new pairs of intersection points may be created during the homotopy, after which one of the original intersection points may annihilate with a newly created one. Thus the original pair cannot in general be tracked continuously while the time parameter moves monotonically from $0$ to $1$; a similar birth-and-death issue appears in \cite{TopAndPhysLinkTheoryDist}. We will circumvent this problem using \lemref{lemma: tracking lemma}. The first tool we need to use for our book-keeping will be the homotopy invariance of the Brouwer degree of a function around $0$. We will therefore briefly recall the following existence and uniqueness theorem defining the degree. This formulation is taken from \textit{Fixed Point Theory} by Granas and Dugundji \cite{FixedPointTheory}.
\begin{thm}[Brouwer degree around $0$]\label{theorem: brouwer degree around a point}
Let $$\mathcal{M}=\{(U,F) \mid U\subseteq \R^n \text{ open, bounded }, F: \overline{U} \to \R^{n} \text{ continuous}, 0\notin F(\partial U)\}.$$
Then there exists a unique function $d: \mathcal{M}\to \Z$ such that
\begin{itemize}  
\item (Normalization) If $0\in U$, then $d(U,\id_{\overline{U}})=1$.
\item (Additivity)  If $ F^{-1}(0)\subseteq U_1\cup U_2 \subseteq U$ with $U_1,U_2$ open and disjoint, then $d(U,F)=d(U_1,F) + d(U_2,F)$.
\item (Homotopy Invariance) If $F_t$ is a homotopy of functions such that for every $t$ we have $0\notin F_t(\partial U)$, then the function $t\mapsto d(U,F_t)$ is constant.
\end{itemize}
\end{thm}
We further need the following two facts about the Brouwer degree. If $F:U \to \R^n$ is $\mathcal{C}^1$ and if $0$ is a regular value of $F$, then $$d(U,F)=\sum_{x\in F^{-1}(0)} \sgn(\det(DF(x))).$$
If $0\notin F(\partial U)$, then $0\notin F(U) \implies d(U,F)=0.$

\begin{lem}\label{lemma: lemma using the homotopy invariance of the Brouwer degree}
Let $\{L_t:[0,1]\to \R^{n+1}\}_{t\in [0,1]}$ and $\{D_t:D^n \to \R^{n+1}\}_{t\in [0,1]}$ be homotopies of a curve and an $n$-dimensional disk respectively, such that
\begin{itemize}
\item $L_0$ and $D_0$ are $\mathcal{C}^1$ embeddings,
\item $L_0$ intersects $D_0$ exactly once at a point in $L_0((0,1))\cap D_0((D^n)^\circ)$, and this intersection is transverse,
\item $L_1$ and $D_1$ are disjoint.
\end{itemize}
Then there is some time $t\in[0,1]$ such that $L_t(\{0,1\})$ and $D_t(D^n)$ intersect or such that $L_t([0,1])$ and $D_t(\partial D^n)$ intersect.
\end{lem}
\begin{proof}
Let $U=(0,1)\times(D^n)^\circ$ and define
$$F_t:\overline{U}\to\R^{n+1},\quad F_t(x,y)=D_t(y)-L_t(x).$$
Since $L_0$ and $D_0$ intersect exactly once in their interiors and since this intersection is transverse, we have $d(U,F_0)=\pm 1$. Since $L_1$ and $D_1$ are disjoint, we have $0\notin F_1(\overline{U})$ and hence $d(U,F_1)=0$. If $0\notin F_t(\partial U)$ for every $t\in[0,1]$, this contradicts the homotopy invariance of the Brouwer degree. Therefore, there is some $t\in[0,1]$ such that $0\in F_t(\partial U)$. Since
$$\partial U=\left(\{0,1\}\times D^n\right)\cup\left([0,1]\times\partial D^n\right),$$
this proves the claim.
\end{proof}

The second tool we will need is the mountain climbing theorem.
\begin{thm}[Mountain Climbing Theorem]\label{theorem: mountain climbing theorem}
Let $\gamma_1,\gamma_2 : [0,1]\to [0,T]$ be continuous such that $\gamma_1(0)=0=\gamma_2(0)$ and $\gamma_1(1)=T=\gamma_2(1)$, and suppose that $\gamma_1$ and $\gamma_2$ are piecewise monotone. Then there exist continuous functions $\gamma'_1, \gamma'_2: [0,1]\to [0,1]$ fixing endpoints such that
$\gamma_1 \circ \gamma'_1 = \gamma_2 \circ \gamma'_2.$
\end{thm}
This result was first stated and proved by Whittaker in 1966 \cite{Whittaker1966MountainClimbing}. It has been strengthened since then, for example in 1969 by Huneke \cite{Huneke1969MountainClimbing}.

\begin{thm}[Disk separation lemma]\label{theorem: disk separation lemma}
Let $D^2$ be a closed $2$-dimensional disk and let $K,L\subseteq D^2$ be closed such that $D^2=K\cup L$ and such that $K\cap \partial D = I_A \cup I_B$ and $L \cap \partial D = I_X\cup I_Y$ are each the disjoint union of two compact intervals of positive length, with $I_A \cap I_X$, $I_B\cap I_X$, $I_A\cap I_Y$, and $I_B\cap I_Y$ each containing exactly one element. Then either $I_A$ and $I_B$ are in the same connected component of $K$ or $I_X$ and $I_Y$ are in the same connected component of $L$.   
\end{thm}
\begin{proof}
Up to homeomorphism we can assume $D^2 = [0,1]^2$ and that $W\coloneqq \{0\}\times [0,1] = I_A$,  $E\coloneqq \{1\}\times [0,1] = I_B$, $S\coloneqq [0,1] \times \{0\}= I_X$,  $N\coloneqq [0,1] \times \{1\}= I_Y$. Toward a contradiction, assume that $W$ and $E$ lie in distinct connected components of $K$ and that $S$ and $N$ lie in distinct connected components of $L$. By Whyburn's lemma \cite{TopologicalAnalysis}, there are disjoint compact sets $K_1,K_2\subseteq K$ such that $K=K_1\cup K_2$, $W\subseteq K_1$, and $E\subseteq K_2$. Similarly, there are disjoint compact sets $L_1,L_2\subseteq L$ such that $L=L_1\cup L_2$, $S\subseteq L_1$, and $N\subseteq L_2$. In particular, $K_1$, $K_2$, $L_1$, and $L_2$ are closed in $D^2$. By Urysohn's Lemma, there are functions $g_K:D^2 \to [0,1]$ and $g_L:D^2 \to [0,1]$ with $g_K(z)=1$ for all $z\in K_1$, $g_K(z)=0$ for all $z\in K_2$ and $g_L(z)=1$ for all $z\in L_1$, $g_L(z)=0$ for all $z\in L_2$. We define the continuous map
$$f:D^2 \to D^2, \quad f(z) = (g_K(z),g_L(z))$$
Note that $z\in K_1 \implies g_K(z)=1$, hence $f$ maps $K_1$ to $E\subseteq K_2$. Similarly $f$ maps $K_2$ to $W\subseteq K_1$ and $L_1$ to $N\subseteq L_2$ and $L_2$ to $S\subseteq L_1$. Any $z\in D^2$ must be in at least one of the sets $K_1$, $K_2$, $L_1$, or $L_2$, and $f(z)$ will be in a set disjoint from it. This means that $f$ cannot have a fixed point, a contradiction to Brouwer's fixed point theorem. We conclude that either $E$ and $W$ are in the same connected component of $K$, or $S$ and $N$ are in the same connected component of $L$. 
\end{proof}

\begin{lem}\label{lemma: hex on a cylinder}
Consider two closed sets $K,L \subseteq S^1 \times [0,1]$ such that both $K$ and $L$ are locally path connected with $K\cup L = S^1 \times [0,1]$, and such that $K\cap (S^{1}\times \{0\}) = I_A \cup I_B$ and $L \cap (S^1 \times \{0\}) = I_X\cup I_Y$ are each the disjoint union of two compact intervals of positive length with $I_A \cap I_X$, $I_B\cap I_X$, $I_A\cap I_Y$, and $I_B\cap I_Y$ each containing exactly one element. Further assume that both $K\cap (S^1 \times \{1\})$ and $L\cap (S^1 \times \{1\})$ are either empty or all of $S^1 \times \{1\}$. Then either $I_A$ and $I_B$ are in the same path component of $K$ or $I_X$ and $I_Y$ are in the same path component of $L$.   
\end{lem}
\begin{proof}
Let $q:S^1 \times [0,1] \to (S^1 \times [0,1]) /\sim$ be the map defined by identifying all points in $S^1 \times \{1\}$ with each other. Then $D^2\coloneqq (S^1 \times [0,1])/\sim$ is a disk. We apply \theoref{theorem: disk separation lemma} to $D^2$ and the sets $q(K)$ and $q(L)$ and conclude that either $q(I_A)$ and $q(I_B)$ are in the same connected component of $q(K)$, or $q(I_X)$ and $q(I_Y)$ are in the same connected component of $q(L)$. Without loss of generality assume the former. Since $K$ is compact and $q(K)$ is Hausdorff, $q|_K:K\to q(K)$ is a quotient map. Since quotients of locally path connected spaces are locally path connected, $q(K)$ is locally path connected and therefore $q(I_A)$ and $q(I_B)$ are in the same path component of $q(K)$. Take a path from $q(I_A)$ to $q(I_B)$ in $q(K)$. Since we assumed that $K\cap (S^1 \times \{1\})$ is either empty or all of $S^1 \times \{1\}$, we can always lift this path to a path from $I_A$ to $I_B$ in $K$.
\end{proof}

We are now able to prove the following result.

\begin{lem}[Tracking Lemma]\label{lemma: tracking lemma}
Let $n\geq 1$, and let $M:S^1\to \R^{n+1}$ and $D:D^n \to \R^{n+1}$ be $\mathcal{C}^1$, such that $M$ and $D$ intersect in exactly two points, and such that these intersections are transverse. Let $M_t$ and $D_t$ be homotopies of $M$ and $D$ respectively, so that $M_t$ never intersects $D_t|_{\partial D^n}$ and so that $M_1$ and $D_1$ are disjoint. Let $\varepsilon>0$. Then there is some $T\in [0,1]$, some continuous $\gamma:[0,1]\to [0,T]$ with $\gamma(0)=0$ and $\gamma(1)=T$, and two continuous maps $x_1,x_2:[0,1]\to S^1$ such that
\begin{itemize}
\item $x_1(0)$ and $x_2(0)$ are the points at which $M$ intersects $D$,
\item $x_1(1)=x_2(1)$,
\item $\dist(M_{\gamma(t)}(x_1(t)),D_{\gamma(t)}), \dist(M_{\gamma(t)}(x_2(t)),D_{\gamma(t)}) <\varepsilon$ for all $t\in [0,1]$.
\end{itemize}

\end{lem}
This lemma tells us that if we have a closed curve and an $n$-dimensional disk in $\R^{n+1}$ that intersect nicely in exactly two spots, and if we have any homotopy splitting the curve and the disk without the curve passing through the boundary of the disk, then we can continuously track points in the parameter space of the curve such that the curve evaluated at these points stays arbitrarily close to the disk, and such that the parameter values eventually merge. This tracking does require us to ``slide the homotopies back and forth through time'', i.e.\ instead of considering $M_t$ and $D_t$, we have to consider $M_{\gamma(t)}$ and $D_{\gamma(t)}$ for some curve $\gamma$.

\begin{proof}
Assume without loss of generality that there is some $c>0$ such that for $t\in [0,c]$ we have $D_t =D$ and $M_t=M$. Let
$$G:[0,1]\times S^1 \to \R, \quad G(t,x) =\dist(M_t(x),D_t),$$
$$S=\{(t,x) \in [0,1]\times S^1 \mid G(t,x)=0\}.$$
Note that $G$ is continuous and hence $S$ is closed. Let $\delta >0$ be such that for all $(t,x)\in [0,1] \times S^1$ with $G(t,x)=0$ we have $(t',x') \in U_{\delta}(t,x) \implies G(t',x')<\varepsilon.$ Let $Z\subseteq [0,1]\times S^1$ be closed and such that
\begin{itemize}
\item $Z$ contains an open neighborhood of $S$
\item $Z$ is contained in the $\delta$-neighborhood of $S$
\item $Z$ and the closure of its complement are locally path connected
\item Whenever there is a path in $Z$ connecting two points in $Z^{\circ}$, there is also a path in $Z^{\circ}$ connecting the same two points
\item Whenever there is a path in the closure of the complement of $Z$ connecting two points in the complement of $Z$, there is a path in the complement of $Z$ connecting the same two points.
\item $\left(\{1\}\times S^1\right) \cap Z =\emptyset$
\item $\left(\{0\}\times S^1\right) \cap Z$ is a disjoint union of two compact intervals $I_A$ and $I_B$ of positive length, each containing one point in $S$.
\end{itemize}
Such a set $Z$ always exists, for example we could cover $S$ with a finite number of closed radius-$\tfrac{\delta}{2}$ balls in a suitable way. Let $I_X$ and $I_Y$ be the disjoint compact intervals such that $I_X^{\circ} \cup I_Y^{\circ} = \left(\{0\}\times S^1\right) \setminus Z$. Let $A$ be the connected component of $Z$ containing $I_A$ and let $B$ be the connected component of $Z$ containing $I_B$. \\
\textbf{Toward a contradiction, assume $\mathbf{A\neq B}$.} By \lemref{lemma: hex on a cylinder}, there is a path $(\alpha,f):[0,1]\to [0,1]\times S^1$  from a point $(0,x) \in I_X^{\circ}$ to a point $(0,y) \in I_Y^{\circ}$ avoiding $Z^{\circ}$. By the properties of $Z$, we can assume that this path avoids $Z$. Since $Z$ is closed, we can perturb the curve $(\alpha,f)$ in the open set $\left([0,1]\times S^1\right) \setminus Z$ so that the curve $\alpha$ is piecewise monotone. Let $T$ be the maximum value of $\alpha$. Without loss of generality assume that $\alpha(\tfrac{1}{2})=T$, else reparametrize $(\alpha,f)$.  We now define $(\alpha_1,f_1),(\alpha_2,f_2):[0,1]\to [0,1]\times S^1$ through $(\alpha_1,f_1)(t)\coloneqq (\alpha,f)(\tfrac{t}{2})$ and $(\alpha_2,f_2)\coloneqq (\alpha,f)(1-\tfrac{t}{2})$. Note that $\alpha_1,\alpha_2:[0,1]\to [0,T]$ agree on endpoints and that both are piecewise monotone.  We may assume that $\gamma\coloneqq \alpha_1 =\alpha_2$, else we appeal to the mountain climbing theorem (\theoref{theorem: mountain climbing theorem}) for the existence of curves $\alpha'_1,\alpha'_2:[0,1]\to[0,1]$, fixing endpoints, with $\gamma\coloneqq \alpha_1\circ \alpha'_1=\alpha_2\circ\alpha'_2$ and then consider $(\gamma,f_1\circ \alpha'_1)$ and $(\gamma,f_2\circ \alpha'_2)$. If $f_1$ and $f_2$ agree before $t=1$, restrict the paths to the first time they agree and reparametrize them. Let $q:\R\to S^1$ be the standard covering. Lift $f_1$ and $f_2$ to paths $\overline{f_1}, \overline{f_2}: [0,1] \to \R$, i.e.\ $q \circ \overline{f_1} = f_1$ and $q\circ \overline{f_2} = f_2$. After interchanging $f_1$ and $f_2$ if necessary, pick the lifts so that $\overline{f_1}(0) < \overline{f_2}(0)<\overline{f_1}(0)+1$ and $\overline{f_1}(1)=\overline{f_2}(1)$. Then define $h:[0,1]^2\to\R$, $h(t,x)=(1-x)\overline{f_1}(t)+ x\overline{f_2}(t)$. For a given $t$, the image of $h(t,\cdot)$ corresponds to the interval between $\overline{f_1}(t)$ and $\overline{f_2}(t)$. Looking at $q(h(t,\cdot)):[0,1]\to S^1$ we get a continuously varying family of curves in $S^1$ with endpoints $f_1(t)$ and $f_2(t)$. Next we define
$$L_t : [0,1] \to \R^{n+1}, \quad L_t(x) = M_{\gamma(t)} ( q(h(t,x))).$$
The curve $L_0$ is simply a reparametrization of a segment of $M$. We further define
$$\tilde{D}_t: D^n \to \R^{n+1}, \quad \tilde{D}_t(y) = D_{\gamma(t)}(y).$$
Note that $L_0$ intersects $\tilde{D}_0=D_0$ exactly once, and that this intersection is transverse. Since $\overline{f_1}(1)=\overline{f_2}(1)$, $L_1$ is the single point $M_{\gamma(1)}(f_1(1))$. Since $(\gamma,f_1)$ avoids $Z\supseteq S$, this point is not in $\tilde{D}_1=D_{\gamma(1)}$. Hence $L_1$ and $\tilde{D}_1$ are disjoint. By \lemref{lemma: lemma using the homotopy invariance of the Brouwer degree}, this is only possible if for some $t\in [0,1]$ we have $L_t(0)=M_{\gamma(t)}(f_1(t))\in \tilde{D}_t$, or $L_t(1)=M_{\gamma(t)}(f_2(t))\in \tilde{D}_t$, or if $L_t$ intersects $\tilde{D}_t(\partial D^n)$. By construction $L_t(0)$ and $L_t(1)$ avoid $\tilde{D}_t$, and we know that $L_t$ and $\tilde{D}_t(\partial D^n)$ are disjoint because we assumed that $M_t$ and $D_t|_{\partial D^n}$ are always disjoint. We have arrived at a contradiction. \textbf{We conclude that our assumption that $\mathbf{A\neq B}$ was wrong and therefore $\mathbf{A=B}$.} By arguments exactly the same as in the proof-by-contradiction, we obtain two paths $(\gamma,x_1), (\gamma,x_2):[0,1]\to Z \subset [0,1]\times S^1$ with $x_1(0)$ and $x_2(0)$ being the two intersections of $M$ and $D$ and with $x_1(1)=x_2(1)$. The fact that $(\gamma,x_1)$ and $(\gamma,x_2)$ have image in $Z$ implies $\dist(M_{\gamma(t)}(x_1(t)),D_{\gamma(t)}), \dist(M_{\gamma(t)}(x_2(t)),D_{\gamma(t)}) <\varepsilon$.

\end{proof}
We now proceed to prove \propref{proposition: lower bound on splitting four component link}

\begin{proof}[Proof of \propref{proposition: lower bound on splitting four component link}]
Let $\varepsilon>0$ and define $\varepsilon' = \tfrac{\varepsilon}{2\pi+4}$. By \lemref{lemma: tracking lemma}, applied with $\varepsilon'$, there is some $T\in [0,1]$, some continuous $\gamma:[0,1]\to [0,T]$ with $\gamma(0)=0$ and $\gamma(1)=T$, and two continuous maps $x_1,x_2:[0,1]\to S^1$ such that
\begin{itemize}
\item $x_1(0)=a$ and $x_2(0)=b$, 
\item $x_1(1)=x_2(1)$,
\item $\dist(M_{\gamma(t)}(x_1(t)),D_{\gamma(t)}), \dist(M_{\gamma(t)}(x_2(t)),D_{\gamma(t)}) <\varepsilon'$ for all $t\in [0,1]$.
\end{itemize}
Assume without loss of generality that $\min\left(\{t\in [0,1] \mid x_1(t)=x_2(t)\}\right)=1$. We define $\left(\tilde{M}_t,\tilde{A}^1_t,\tilde{A}^2_t, \tilde{C}_t\right)=\tilde{H}_t = H_{\gamma(t)}$, and let $\tilde{D}_t$ be the center of mass spanning disk of $\tilde{C}_t$. For any $t\in [0,1]$, $x_1(t)$ and $x_2(t)$ partition $S^1$ into two parts, say $I_t^1$ and $I_t^2$. We define $\tilde{M}_t^1$ as the concatenation of $\tilde{M}_t|_{I^1_t}$ and the straight line segment between $\tilde{M}_t(x_1(t))$ and $\tilde{M}_t(x_2(t))$. Analogously we define $\tilde{M}_t^2$ as the concatenation of $\tilde{M}_t|_{I^2_t}$ and the straight line segment between $\tilde{M}_t(x_1(t))$ and $\tilde{M}_t(x_2(t))$. These are continuously varying families of curves. Either $\tilde{M}^1_1$ or $\tilde{M}^2_1$ becomes a single point. Without loss of generality assume that $\tilde{M}^1_1$ becomes a single point. Since $\tilde{M}^1_0 $ was linked with $\tilde{A}^1_0$, and since $\tilde{M}^1_1$ is a point and hence not linked with $\tilde{A}^1_1$, there must be some $t\in [0,1]$ such that $\tilde{M}^1_t$ and $\tilde{A}^1_t$ intersect. This intersection must happen on the straight line segment between $\tilde{M}_t(x_1(t))$ and $\tilde{M}_t(x_2(t))$, as $\tilde{M}_t$ and $\tilde{A}^1_t$ stay disjoint throughout. Let $P_1:=\tilde{M}_t(x_1(t))$ and $P_2:=\tilde{M}_t(x_2(t))$, then there is some $Q \in \tilde{A}^1_t$ on the line segment $\overline{P_1P_2}$ with $||P_1-Q||,||P_2-Q||\geq 1$. This implies $||P_1-P_2||\geq 2$. Further, we have $\dist(P_1,\tilde{D}_t) <\varepsilon'$ and likewise $\dist(P_2,\tilde{D}_t)<\varepsilon'$. By \propref{proposition: main inequality with epsilons.}, we get
$$\len(C_{\gamma(t)})=\len(\tilde{C}_t) \geq 2\pi  + 2||P_1-P_2|| - (2\pi+4)\varepsilon' \geq 2\pi +4 -\varepsilon.$$
Since $\varepsilon>0$ was arbitrary, we get $\sup_{t\in [0,1]} \len(C_t) \geq 2\pi+4$.

\end{proof}

\subsection{Constructing an isotopically Gordian split link}

Next we exhibit a thickly embedded split link $L$ fulfilling the conditions of \propref{proposition: lower bound on splitting four component link}, which has an initial length of less than $8\pi+6$. This means that $L$ cannot be split through a thick homotopy. A fortiori, $L$ cannot be split through a thick isotopy. So $L$ is a Gordian split link both as a thick and as a thickly embedded link. We construct the link $L=(M,A^1,A^2,C)$ such that the main curve $M$ is in the $XY$-plane, the auxiliary curves $A^1$ and $A^2$ are in the $XZ$-plane, and the center curve is in the $YZ$-plane.  The following plots, as well as numerical calculations of some of the curves' lengths, were created and conducted with Mathematica.

\begin{figure}[H]
  \centering
  \begin{minipage}[b]{0.45\textwidth}
    \centering
    \includegraphics[width=\textwidth]{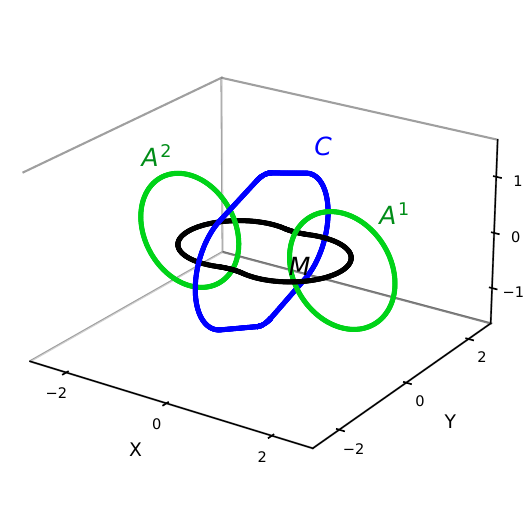}
    \captionof{figure}{A plot of the thickly embedded Gordian split link $L$}
    \label{fig:3d-curve}
  \end{minipage}\hfill
  \begin{minipage}[b]{0.45\textwidth}
    \centering
    \includegraphics[width=\textwidth]{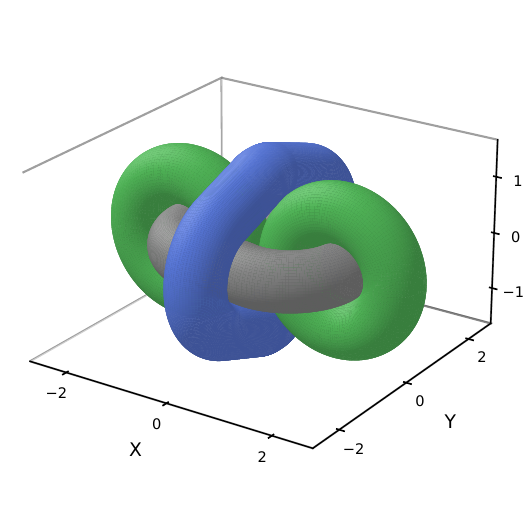}
    \captionof{figure}{Plot of the radius-$\tfrac{1}{2}$ tubes around the link \(L\).}
    \label{fig:3d-tube}
  \end{minipage}
\end{figure}

To see how our Gordian split link is defined, we now look at the three different coordinate planes. In the plots of the coordinate planes, we will always plot the curves, as well as forbidden areas, i.e.\ the $1$-neighborhoods of the curves in other coordinate planes. We will also plot and reference selected points to make following the definition of the curves easier. The points are named $A$ through $X$, so some points share a name with our curves $M$ and $C$. It will always be obvious which is meant, so there should be no confusion. We define the curves in sequence. First we define the auxiliary curves $A^1$ and $A^2$, then the main curve $M$ and finally the center curve $C$. 

\begin{figure}[H]
        \centering
        \includegraphics[width=0.5\textwidth]{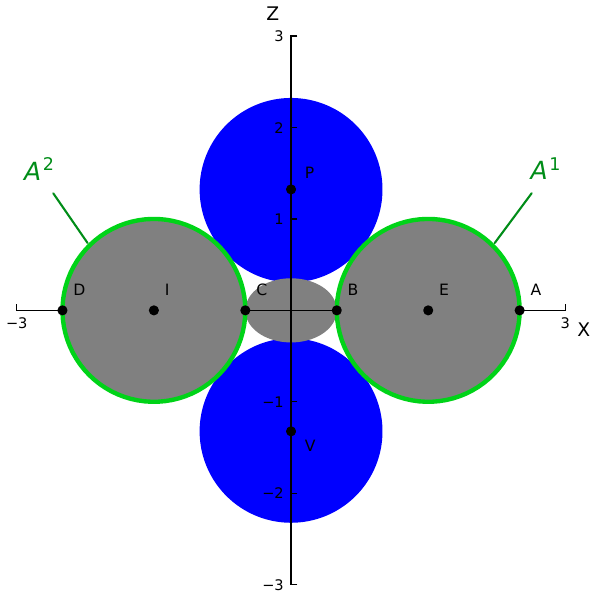} 
        \caption{A plot of the $XZ$-plane. The auxiliary curves $A^1$ and $A^2$ are plotted in green, the forbidden areas belonging to main curve and center curve are plotted in gray and blue respectively.}
        \label{XZ plot}
\end{figure}

\begin{figure}[H]
        \centering
        \includegraphics[width=0.5\textwidth]{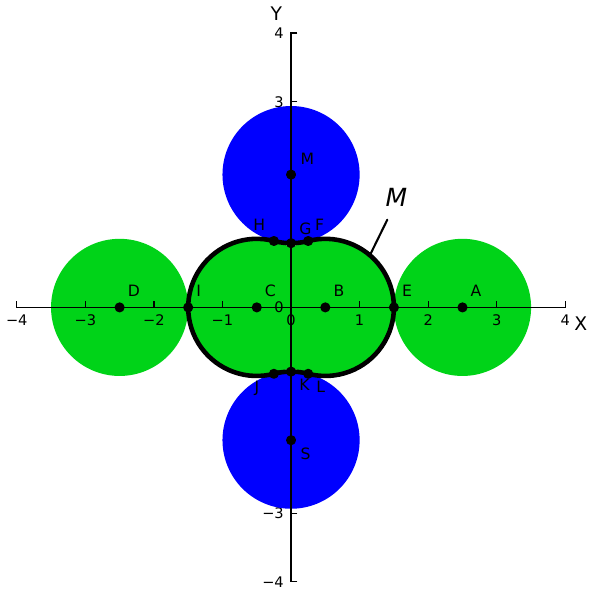} 
        \caption{A plot of the $XY$-plane. The main curve $M$ is plotted in black, the forbidden area belonging to the auxiliary curves $A^1$ and $A^2$ is plotted in green and the forbidden area belonging to center curve $C$ is plotted in blue.}
        \label{XY plot}
\end{figure}

\begin{figure}[H]
        \centering
        \includegraphics[width=0.5\textwidth]{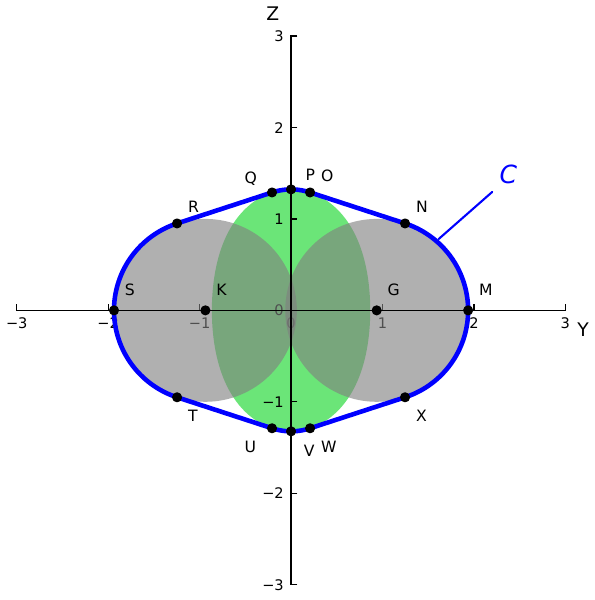} 
        \caption{A plot of the $YZ$-plane. The center curve $C$ is plotted in blue, the forbidden area belonging to the auxiliary curves $A^1$ and $A^2$ is plotted in green and the forbidden area belonging to the main curve $M$ is plotted in gray.}
        \label{YZ plot}
\end{figure}

The two auxiliary curves $A^1$ and $A^2$ are unit circles in the $XZ$-plane with centers $E:=\left(\tfrac{3}{2},0,0\right)$ and $I:=\left(-\tfrac{3}{2},0,0\right)$. The main curve $M$ is defined as follows: In the $XY$-plane, take the unit circles with centers $B:=\left(\tfrac{1}{2},0,0\right)$, $M:=\left(0,\tfrac{\sqrt{15}}{2},0\right)$, $C:= \left(-\tfrac{1}{2},0,0\right)$, and $S:=\left(0,-\tfrac{\sqrt{15}}{2},0\right)$. The circles around $M$ and $B$ touch in exactly one point $F$. Similarly, we can define $H$, $J$, and $L$. Let $G:= \left(0,\tfrac{\sqrt{15}}{2}-1,0\right)$ and $J:=\left(0,-\left(\tfrac{\sqrt{15}}{2}-1\right),0\right)$. The main curve $M$ is defined as the concatenation of four unit circle segments (see \figref{XY plot}). Let $\alpha = \angle FMG$, then $\alpha = \tfrac{\pi}{2}-\arctan(\sqrt{15}) \approx 0.2527$. We get $\len(M) = 2\pi +8\alpha = 6\pi - 8\arctan(\sqrt{15}) \approx 8.3046$. To define the center curve $C$, we parametrize the forbidden green region, which is the intersection of the open solid torus of radius $1$ around either auxiliary curve with the $YZ$-plane. This region, which we will denote by $\mathcal{G}$, is given by $$\left( \left(\tfrac{3}{2}\right)^2 + z^2 + y^2 \right)^2 - 4 \left(\left(\tfrac{3}{2}\right)^2 + z^2\right) < 0,$$ and its boundary $\partial \mathcal{G}$ is a closed curve implicitly defined through $$\left(\left(\tfrac{3}{2}\right)^2 + z^2 + y^2 \right)^2 - 4 \left(\left(\tfrac{3}{2}\right)^2 + z^2\right) = 0.$$
There is a unique straight line segment in the quadrant $y,z\geq 0$ connecting a point on the unit circle around $G$ and a point on the boundary curve $\partial\mathcal{G}$ while being tangent to both the unit circle and $\partial \mathcal{G}$. We call the start of this line segment $N$ and the end of it $O$. Let $P=\left(0,0,\tfrac{\sqrt{7}}{2}\right)$, and recall that $M=\left(0,\tfrac{\sqrt{15}}{2},0\right)$. We define the first quarter of our center curve $C$ as the concatenation of the circle segment from $M$ to $N$, the straight line segment from $N$ to $O$, and the segment of the boundary curve $\partial\mathcal{G}$ from $O$ to $P$. The other three quarters of $C$ are defined by reflecting the first segment from $M$ to $P$ along both the $Y$-axis and the $Z$-axis. Mathematica was not able to solve for the coordinates of $N$ and $O$ exactly, so they were calculated numerically instead. The angle $\beta=\angle MGN$ was calculated as $\beta\approx 1.2554$. The length of the straight line segment $s$ from $N$ to $O$ was calculated as $\len(s)\approx 1.0925$. The length of the segment $b$ from $O$ to $P$ along $\partial \mathcal{G}$ was numerically calculated as $\len(b)\approx 0.2116$. The length of $C$ is therefore $\len(C)=4\left(\beta + \len(s)+\len(b)\right) \approx 10.2380$. Finally, we get
\begingroup
\bfseries\boldmath
\begin{align*}
\len(L) = &2\pi + 2\pi + \left(6 \pi - 8\arctan(\sqrt{15})\right) + 4\left(\beta + \len(s)+\len(b)\right)\\
\approx & 31.1090 < 31.1327 \approx 8\pi +6.
\end{align*}
\endgroup
We have constructed $L$ to be a thick link, but have not yet shown that $L$ is thickly embedded. The open radius-$\tfrac{1}{2}$ normal disk bundles around $A^1$, $A^2$ and $M$ are obviously embedded. To see that the open radius-$\tfrac{1}{2}$ normal disk bundle around $C$ is embedded, we need to check that the curvature of the segment $b$ is bounded above by $2$. This is true. In fact, the curvature of all of $\partial \mathcal{G}$ is bounded above by $2$. We omit a proof as it is not very interesting. 
Up to that detail, we have proven the following.

\begin{thm}[A $4$-component Gordian split link exists]\label{theorem: A Gordian split link exists}
There exist $4$-component thickly embedded links that are Gordian split links as thick links (and hence a fortiori as thickly embedded links).
\end{thm}
The difference between the lower bound we derived in \subsectref{subsection: A lower bound on ropelength needed to split some thick links} and the length of our constructed configuration $L$ is less than $0.08\%$ of $\len(L)$.  While it is unlikely that $L$ can be shortened much, it is highly likely that the lower bound of $8\pi+6$ from \propref{proposition: lower bound on splitting four component link} can be significantly improved.  A careful examination of the proof of \propref{proposition: lower bound on splitting four component link} shows that we never needed to assume that the distances between $A^1$ and $C$ and between $A^2$ and $C$ stay at least $1$ throughout the homotopy. In fact, we have not even used that $A^1$ and $A^2$ stay disjoint from $C$. We only needed that $(A^1,A^2,M)$ and $(M,C)$ stay thick links throughout the homotopy of $L$.

\section{Homotopically Gordian pairs exist in every link homotopy class with \texorpdfstring{$2$}{2} components}\label{Homotopically Gordian pairs exist in every link homotopy class with 2 components}

\subsection{Circle decompositions - the idea}

We first sketch the idea for our construction, as well as the proof, informally. Consider a two-component thick link $L$ made up of unit circle segments as in \figref{figure: 2-component gordian unlink} and parametrized as shown in \figref{figure: parametrization of circle decomposition}.

\begin{figure}[H]
  \centering
  \begin{minipage}[t]{0.45\textwidth}
    \vspace{0pt}%
    \centering
    \includegraphics[width=\textwidth]{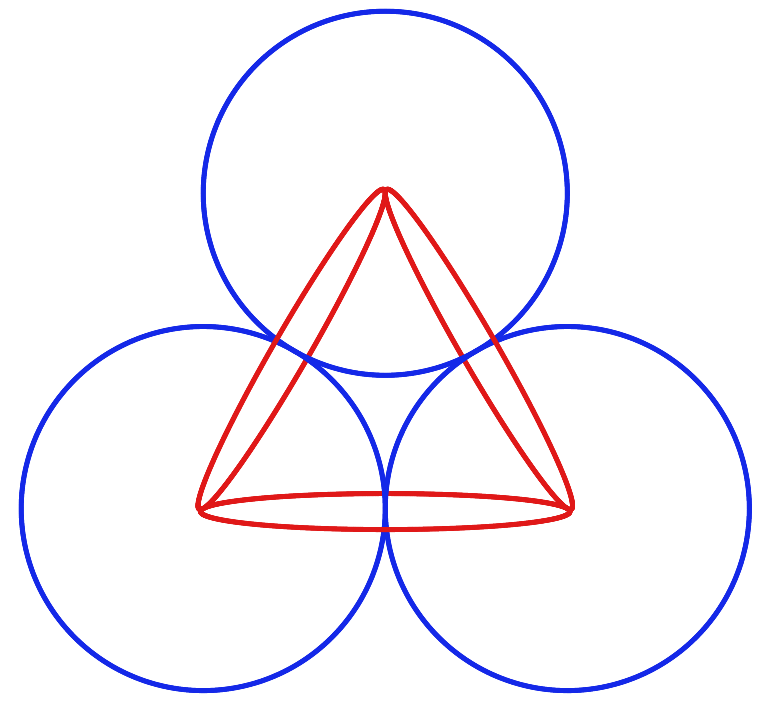}
    \captionof{figure}{A conjoined circle decomposition: The image of the blue curve is the union of three coplanar, pairwise tangent unit circles. The image of the red curve is the union of three pairwise tangent unit circles, each piercing the centers of two blue circles.}
    \label{figure: 2-component gordian unlink}
  \end{minipage}\hfill
  \begin{minipage}[t]{0.45\textwidth}
    \vspace{0pt}%
    \centering
    \includegraphics[width=\textwidth]{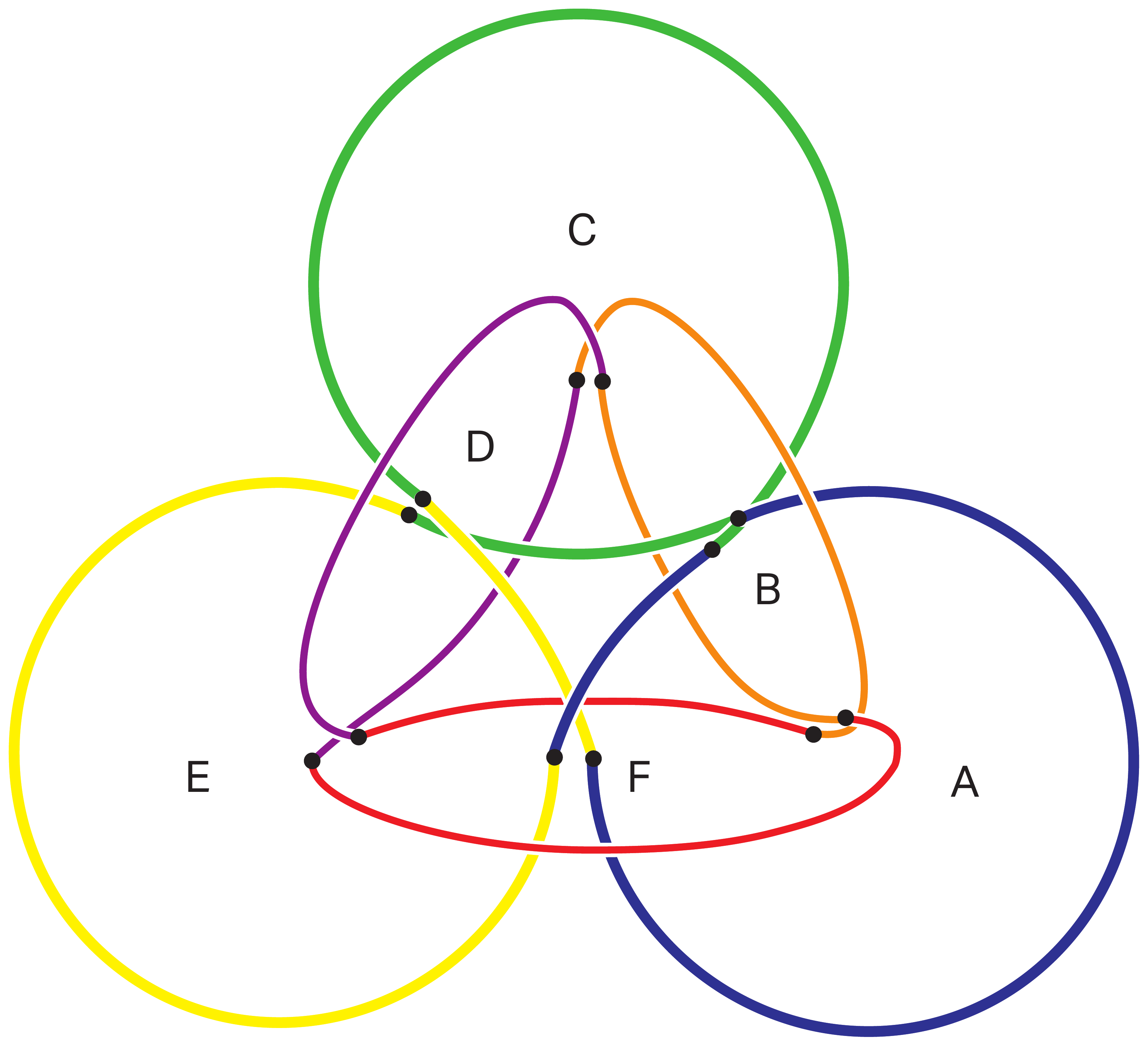}
    \captionof{figure}{A perturbation of a conjoined circle decomposition: This figure shows a thick link $L'$ close to $L$ to help understand how $L$ is parametrized.}
    \label{figure: parametrization of circle decomposition}
  \end{minipage}
\end{figure}

In the original configuration, there are two strands going through each circle's center. After a small perturbation, for every circle, pick one point each on both of the strands through that circle minimizing their respective distance. This gives a set of two points for every circle. These sets are marked $A$, $B$, $C$, $D$, $E$, $F$ in \figref{figure: parametrization of circle decomposition}. We now split up the two link-components into $6$ curves each, along the points making up the sets $A$, $B$, $C$, $D$, $E$, $F$, and pair them up as indicated by the color scheme. Let us calculate the sum $S_{\text{blue}}$ of the lengths of the blue segments. We define the closed curve $C_{\text{blue}}$ by connecting the two blue components by a straight line segment through the $2$ points in $B$ and by a straight line segment through the $2$ points in $F$. Then the two strands piercing the blue circle will pierce the center-of-mass spanning disk of $C_{\text{blue}}$ in (at least) one point each, say in the \textit{breaking points} $p$ and $q$, where $\{p,q\}=A$. We \textbf{hope to have}
$$\len(C_{\text{blue}}) \geq 2\pi + 2||p-q|| = 2\pi + 2\diam(A),$$
$$S_{\text{blue}} = \len(C_{\text{blue}}) - \diam(B) -\diam(F) \geq 2\pi + 2\diam(A) - \diam(B) - \diam(F).$$
Doing the same thing for all other colors and summing the lengths we would then get that all the terms involving diameters of sets $A,B,C,D,E,F$ cancel, so that the only thing remaining is
$$\len(L') \geq 12\pi.$$
This would then show that any thick link sufficiently close to the original configuration is at least as long, so the original configuration seems to be a local minimum. Besides the many missing details, there are two problems with this proof. First, even if our arguments are correct, we only get that the initial configuration is a local minimum, which is not yet sufficient to conclude that it is a sink. Secondly, we do not know that the straight line segment keeps distance at least $1$ to the points $p$ and $q$ intersecting the spanning disk of $C_{\text{blue}}$. But this was a condition implicitly used when we claimed $\len(C_{\text{blue}}) \geq 2\pi + 2||p-q||$. We will see that we can make adjustments to the argument to get
\begin{align*}
\len(L') \geq 12\pi &+2\diam(A) - 2 \varphi(\diam(A))\\
&+2\diam(B) - 2 \varphi(\diam(B))\\
&+2\diam(C) - 2 \varphi(\diam(C))\\
&+2\diam(D) - 2 \varphi(\diam(D))\\
&+2\diam(E) - 2 \varphi(\diam(E))\\
&+2\diam(F) - 2 \varphi(\diam(F)),
\end{align*}
where near $0$ we have that $\varphi(x) \sim x$ to first order, but with $\varphi(x)>x$ for $x>0$. This adjusted argument is no longer strong enough to show that the original configuration was a local minimum. In some ways, the reason why the argument fails is that there are too many \textit{breaking points}. We can fix this by augmenting $L$.\\

\begin{figure}[H]
        \centering
        \includegraphics[width=0.5\textwidth]{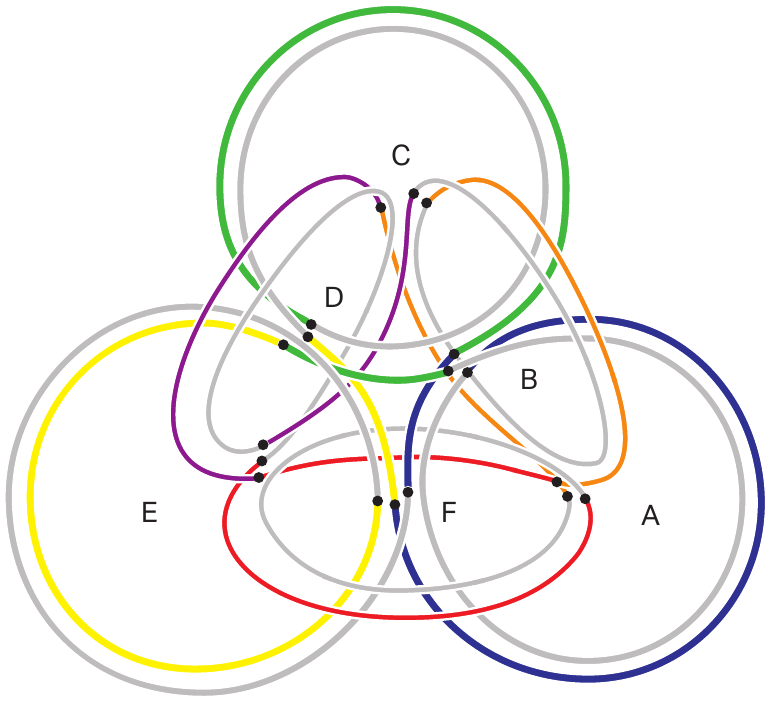} 
        \caption{A perturbation of an augmented conjoined circle decomposition.}\label{figure: circle decomposition augmented}
\end{figure}
In the configuration shown in \figref{figure: circle decomposition augmented}, we go through every one of the original $6$ circles one additional time. This only introduces one new breaking point for each circle, which corresponds to one new term of $-\varphi(\diam(\cdot))$ for each of the sets $A,B,C,D,E,F$. At the same time, each of the circles introduces a term of $2\diam(\cdot)$ for the sets $A,B,C,D,E,F$. This gives the inequality
\begin{align*}
\len(L') \geq 24\pi &+4\diam(A) - 3 \varphi(\diam(A))\\
&+4\diam(B) - 3 \varphi(\diam(B))\\
&+4\diam(C) - 3 \varphi(\diam(C))\\
&+4\diam(D) - 3 \varphi(\diam(D))\\
&+4\diam(E) - 3 \varphi(\diam(E))\\
&+4\diam(F) - 3 \varphi(\diam(F)).
\end{align*}
Since $\varphi(x)\sim x$ to first order, for small enough diameters this inequality yields
$$\len(L') \geq 24\pi$$
with equality if and only if $\diam(A)=\diam(B)=\dots=\diam(F)=0$. That is, we have $\len(L') \geq 24\pi$ with equality if and only if $L'$ also decomposes into circles.
\pagebreak

\subsectref{subsection: Circle decompositions - the details} is devoted purely to filling in the missing details.  \textbf{Therefore one might skip \subsectref{subsection: Circle decompositions - the details} to read the short proofs of the final results in \subsectref{subsection: A Gordian unlink and other examples}.}

\subsection{Circle decompositions - the details}\label{subsection: Circle decompositions - the details}
We now formally define what we mean by a circle decomposition and a conjoined circle decomposition. Throughout the rest of this paper $\frown$ will denote concatenation of curves.
\begin{defin}[circle decomposition]
Let $C$ be a closed curve such that the domain of $C$ decomposes into $I_1 \cup I_2 \cup \dots \cup I_n$, where each $I_i$ is a compact interval of positive length with $\max(I_i)=\min(I_{i+1})$,  for all $i$ (here indices are understood to be cyclic and the domain is seen as an interval with endpoints identified). We define $a_i$ via $[a_i,a_{i+1}]=I_i$. Let $\{A_j\}_{j =1}^m$ and $\{\alpha_l\}_{l=1}^k$ each be a partition of $\{1,2,\dots,n\}$ into non-empty sets. We say that $\left(C,\{a_i\}_{i=1}^n,\{A_j\}_{j=1}^m,\{\alpha_l\}_{l=1}^k \right)$ is a \textit{circle decomposition} of $C$ if for all $i_1,i_2\in \{1,\dots,n\}$ and $j \in \{1,\dots,m\}$:
\begin{itemize}
\item There is a bijection $\{1,2,\dots,c\} \to A_j$, $k\mapsto i_k$ and some $s_1,\dots,s_c \in \{1,-1\}$, such that $C^j := C_{i_1}^{s_1} \frown C_{i_2}^{s_2}\frown \dots \frown  C_{i_c}^{s_c}$ is a curve of length $2\pi$ parametrizing a unit circle. Here $C_{i_k}^1 := C|_{I_{i_k}}$ and $C_{i_k}^{-1}$ is defined to be the curve obtained by inverting the orientation of $C|_{I_{i_k}}$.
\item $C(a_{i_1})=C(a_{i_2}) \iff \text{ there is some } l \text{ with } i_1,i_2 \in \alpha_l$
\end{itemize}
\end{defin}
The length of a circle decomposition $\left(C,\{a_i\}_{i=1}^n,\{A_j\}_{j=1}^m,\{\alpha_l\}_{l=1}^k \right)$ is equal to $2m\pi$. 
\begin{rem}
From here on we will write circle decompositions just as\\
$(C,\{a_i\},\{A_j\},\{\alpha_l\})$, without specifically referencing the index-sets. We want to be able to talk about multiple circle decompositions at once. For consistency and readability, we state some conventions that will be used throughout the rest of this section.
\begin{itemize}
\item We implicitly use the definition $K_j:=\bigcup_{i\in A_j} C(I_i)$ and the definition of $z_j \in \R^3$ as the center of the circle $K_j$.
\item We use $i$ to index points or intervals in parameter space, where $I_i=[a_i,a_{i+1}]$. We call the points $a_i$ breaking points.
\item We use $j$ to index circles of the circle decomposition:  either for the collection of (indices representing) intervals $A_j$ that map to a circle, or for the closed curve $C^j$ parametrizing $K_j$.
\item We will use $l$ to index sets $\alpha_l$ of indices representing points $a_i$. We will call the sets $\alpha_l$ point clusters, as they are made up of indices representing a subset of $\{a_i\}$ with a common image point on $C$. Because of the way $\alpha_l$ is defined, we can define $C(\alpha_l):=C(a_i)$ for some (hence every)  $i\in \alpha_l$.
\end{itemize}
\end{rem}

\begin{defin}
Let $(C,\{a_i\},\{A_j\},\{\alpha_l\})$ be a circle decomposition.
We say that two indices $i_1,i_2 \in \alpha_l$ are \textit{distinguishable} if there are small neighborhoods $H_{i_1}$ around $a_{i_1}$ and $H_{i_2}$ around $a_{i_2}$ such that $C|_{H_{i_1}}$ and $C|_{H_{i_2}}$ intersect only at $C(\alpha_l)$.
\end{defin}
Equivalently, two indices are distinguishable if their local passages through $C(\alpha_l)$ have no one-sided geometric circle arc in common.

\begin{defin}
Let $(C,\{a_i\},\{A_j\},\{\alpha_l\})$ be a circle decomposition. If $a_i$ is on the boundary of one of the intervals $I_{i'}$ with $i'\in A_j$, then we say that $a_i$ is \textit{along} the circle $C^j$. We say that $\alpha_l$ is \textit{along} the circle $C^j$ if there is some $i\in \alpha_l$ with $a_i$ along $C^j$.
\end{defin}
It will be useful to define the notion of a transverse intersection in the non-differentiable setting.
\begin{defin}[transverse intersection]
Let $C:[0,1]\to \R^{n+1}$ be a curve and let $D:[0,1]^n \to \R^{n+1}$ be a disk. We define 
$$D-C:[0,1]\times[0,1]^n\to \R^{n+1}, \quad (D-C)(x,y)=D(y)-C(x).$$ We say that $C$ intersects $D$ \textit{transversely} at a point $p=D(y)=C(x)$ if there is an open disk $U\subseteq [0,1]^n$ around $y$ and an open interval $I\subseteq [0,1]$ around $x$ such that $D|_U$ and $C|_I$ only intersect in $p$ and such that the Brouwer degree of $(D-C)|_{I\times U}$ around $0$ is $\pm 1$.
\end{defin}

Next we define conjoined circle decompositions.
\begin{defin}
Let $(C_1,\{a_i\},\{A_j\},\{\alpha_l\})$ be a circle decomposition of $C_1$ and let
$(C_2, \{b_i\},\{B_j\},\{\beta_l\})$ be a circle decomposition of $C_2$. We say that $C_1 \cup C_2$ has a \textit{conjoined circle decomposition} if
\begin{itemize}
\item $\dist(C_1,C_2)=1,$
\item $K_j$ is a circle of $C_1$ $\implies$ there is some $\beta_l$ such that $z_j = C_2(\beta_l),$ and every intersection of  $C_2$ and $D_j$ is transverse, 
\item $K_j$ is a circle of $C_2$ $\implies$ there is some $\alpha_l$ such that $z_j = C_1(\alpha_l),$ and every intersection of  $C_1$ and $D_j$ is transverse,
\item for every $\alpha_l$ there is some circle $K_j$ of $C_2$ such that $C_1(\alpha_l)=z_j,$
\item for every $\beta_l$ there is some circle $K_j$ of $C_1$ such that $C_2(\beta_l)=z_j.$
\end{itemize}
The second and third conditions guarantee the existence of some functions $h^A$, $h^B$ defined as follows:
$$h^B(j) \text{, such that } C_2(\beta_{h^B(j)})=z_j, \quad h^A(j) \text{, such that } C_1(\alpha_{h^A(j)})=z_j.$$
So $h^A$ takes a $B$-index $j$ of some circle $K_j$ of $C_2$ to the $A$-index $l$ of the point cluster $\alpha_l$ which is mapped to $z_j$, and analogously for $h^B(j)$. The \textit{linking scheme} of $C_1$ and $C_2$ is defined to be the data for which indices $j$ belonging to the circle decomposition of $C_1$ and $j'$ belonging to the circle decomposition of $C_2$ the circles $K_j$ and $K_{j'}$ are non-trivially linked. The data of the conjoined circle decomposition are the data of the individual circle decompositions along with the functions $h^A$ and $h^B$, and the linking scheme.
\end{defin}

Every circle of one of the circle decompositions has its center pierced by the other curve, and this only ever happens transversely. The arc-length parametrization of both curves is piecewise $\mathcal{C}^1$ as it consists of circle segments. Further, the transitions between circle segments happen only when a curve pierces a unit disk of a unit circle of the other curve. Because of the condition $\dist(C_1,C_2)=1$, we know that each curve pierces these disks only at the center, and it does so perpendicularly. Therefore the arc-length parametrization of both curves is actually $\mathcal{C}^1$.

\begin{rem}
Let $(C_1,\{a_i\},\{A_j\},\{\alpha_l\}), (C_2, \{b_i\},\{B_j\},\{\beta_l\})$ be a conjoined circle decomposition. When we have a conjoined circle decomposition, we will refer to indices $i$, $j$, and $l$ belonging to the circle decomposition of $C_1$ as $A$-indices and to the indices $i$, $j$ and $l$ belonging to the circle decomposition of $C_2$ as $B$-indices. This way of handling the many different (kinds of) indices implicitly, without defining and naming $6$ different index sets to keep track of, will hopefully be easier to follow than the alternative.
\end{rem}
For technical reasons, we will need to consider a smaller class of conjoined circle decompositions.

\begin{defin}
Let $(C_1,\{a_i\},\{A_j\},\{\alpha_l\}), (C_2, \{b_i\},\{B_j\},\{\beta_l\})$ be a conjoined circle decomposition. We say that the conjoined circle decomposition is \textit{suitable}, if
\begin{itemize}

\item  Every circle $K_j$ of one of the curves is linked with exactly two circles $K_{j_1},K_{j_2}$ of the other curve,
\item for every $A$-index $l$: $\#\{j \text{ } A\text{-index } \mid \alpha_l \text{ along } C_1^j\}<2\#(h^A)^{-1}(l)$,
\item for every $B$-index $l$: $\#\{j \text{ } B\text{-index } \mid \beta_l \text{ along } C_2^j\}<2\#(h^B)^{-1}(l)$,
\item for every $\alpha_l$, there are two distinguishable $i_1,i_2\in \alpha_l$,
\item for every $\beta_l$, there are two distinguishable $i_1,i_2\in \beta_l$.
\end{itemize}
\end{defin}
The first condition is not strictly necessary for our arguments, but it will make the book-keeping a lot easier. Note that we really do mean that every circle belonging to one curve is linked with two distinct circles of the other curve, not with two circles with distinct indices. This is important as we can have $K_{j_1} = K_{j_2}$ even if $j_1\neq j_2$. The first condition also implies that two distinct geometric circles of one component have different linking partners. Indeed, a unit circle linked with two distinct unit circles of the other component must pass through their centers and meet their spanning disks perpendicularly, and there is at most one unit circle with this property. The next two suitability conditions can be rephrased as ``The number of circles broken up by any point cluster is less than two times the number of circles whose center is that point cluster''. This is the main condition needed to make our inequalities work out. The last two conditions are so that we can use \lemref{lemma: breaking point clustering is nice}. We can now state the main theorem of this chapter.

We equip the set of $2$-component links with the topology induced by the supremum norm on maps from $S^1\sqcup S^1$ to $\R^3$.

\begin{thm}\label{theorem: main theorem of chapter 4}
Let $S$ be the set of $2$-component links in some link homotopy class, that all have suitable conjoined circle decompositions with the same given index sets $\{A_j\}$, $\{\alpha_l\}$, $\{B_j\}$, $\{\beta_l\}$, the same index functions $h^A,h^B$ and the same linking scheme. Then $S$ is a closed, strictly locally minimizing set.
\end{thm}

The main tool in proving \theoref{theorem: main theorem of chapter 4} will be finding a lower bound for the length of \textit{almost closed} curves avoiding some set of points. In the following lemmas, $U_r(\cdot)$ will refer to the open $r$-neighborhood of a set or point.

\begin{lem}\label{lemma: avoiding two circles in the plane}
Let $D_1$ and $D_2$ be two open disks in $\R^2$, with radii $r_1,r_2$ and centers $z_1,z_2$ respectively, with $r_1,r_2 \in [\tfrac{1}{2},1]$ and $||z_1-z_2||<\tfrac{1}{2}$. For all $a,b \in \R^2 \setminus \left(D_1 \cup D_2\right)$ with $||a-b||<\tfrac{1}{2}$ and all $m\in D_1 \cap D_2$ not collinear with $a$ and $b$, there exists a curve $C$ from $a$ to $b$, such that $$\len(C) \leq \arcsin(2||a-b||)-||a-b||,$$ obtained by some radial projection of the straight-line segment $\overline{ab}$ from $m$. 
\end{lem}
\begin{proof}
Consider the straight line segment $\overline{ab}$. If $\overline{ab}$ avoids $D_1 \cup D_2$, then we take $C= \overline{ab}$. If $\overline{ab}$ intersects $D_1 \cup D_2$ then we define $C$ by radially projecting (from $m$) the part of $\overline{ab}$ that lies between the first and the final intersection of $\overline{ab}$ and $D_1 \cup D_2$ onto $\partial(D_1\cup D_2)$ and by leaving the rest of $\overline{ab}$ the same. Then $C = l_1 \frown A \frown l_2$ where $l_1,l_2$ are straight line segments which are part of $\overline{ab}$ and where $A$ is the path on the boundary. 
There are three cases, either $A$ is a circular arc, the concatenation of two circular arcs, or the concatenation of three circular arcs. Checking that $\len(C) \leq \arcsin(2||a-b||)-||a-b||$ in all three cases is elementary.
\end{proof}

\begin{lem}\label{lemma: lemma about length of curve avoiding 1-neighborhood of two points}
Let $x,y,m \in \mathbb{R}^3$, not necessarily distinct, with $||x-y||,||x-m||,||y-m||< \tfrac{1}{2}$. Let $a,b \in \mathbb{R}^3 \setminus U_1(\{x,y\})$ with $||a-b|| < \tfrac{1}{2}$.
Let $P$ be a plane through $a$, $b$ and $m$. Then there is a curve $C \subseteq P$ from $a$ to $b$, avoiding $U_1(\{x,y\})$ with $\len(C) \leq \arcsin(2||a-b||)-||a-b||$ and such that $C$ is obtained by radial projection of the straight line segment $\overline{ab}$ away from $m$.  
\end{lem}

\begin{proof}
Toward a contradiction, assume $m \in \overline{ab}$. Then $||a-x||\leq ||a-m||+||m-x|| \leq ||a-b|| + ||m-x|| <1$, which contradicts $a\notin U_1(\{x,y\})$. So $m\notin \overline{ab}$. This means that if $a$, $b$ and $m$ are collinear, then either $a\in \overline{mb}$ or $b\in \overline{am}$. Without loss of generality assume that $a\in \overline{mb}$. Toward a contradiction, assume there is some $c\in \overline{ab}$ with $c\in U_1(x)$. Since $U_1(x)$ is convex and since $a\in \overline{mc}$, this implies $a\in U_1(x)$, which is not the case. We conclude that $\overline{ab}$ and $U_1(x)$ are disjoint. Analogously we can conclude that $\overline{ab}$ and $U_1(y)$ are disjoint. We define $C=\overline{ab}$. If $a$, $b$ and $m$ are not collinear, then by taking the intersection of $U_1(x)$ and $U_1(y)$ with $P$, the problem reduces to \lemref{lemma: avoiding two circles in the plane}.
\end{proof}

\begin{prop}\label{proposition: length of almost closed curves}
Let $C:[0,1] \to \R^3$ be a curve. Let $\bar{C}=C\frown S$ be the closed curve obtained by concatenating $C$ with the straight line segment from $a=C(1)$ to $b=C(0)$ with $||a-b||<\tfrac{1}{2}$. Let $m$ be a point in the convex hull of $\bar{C}$, $K$ the $m$-cone of $\bar{C}$ and $x,y \in K$ with $1\leq \dist(x,C),\dist(y,C)$ and $||x-y||,||x-m||,||y-m||<\tfrac{1}{2}.$ Let $\varphi(t)=\arcsin(2t)-t$. Then
$$\len(C) \geq 2\pi +2||x-y|| -\varphi(||a-b||)$$
\end{prop}

\begin{proof}
Let $P$ be a plane containing $m$, $a$ and $b$. Then, by \lemref{lemma: lemma about length of curve avoiding 1-neighborhood of two points}, there is a curve $\tilde{C}$ in $P$ connecting $a$ to $b$ obtained by projecting $S$ radially away from $m$ in $P$ with $\len(\tilde{C}) \leq \varphi(||a-b||)$ and with $1\leq \dist(x,\tilde{C}),\dist(y,\tilde{C})$. Let $\tilde{K}$ be the $m$-cone of $C\frown \tilde{C}$. Note that $m\in \conv(C\frown S)\subseteq \conv(C\frown \tilde{C})$ and that $x,y \in \tilde{K}$. Since $C\frown \tilde{C}$ avoids $U_1(\{x,y\})$, we can apply \propref{proposition: two-point arbitrary radius inequality} to get
$$\len(C) + \len(\tilde{C}) = \len(C\frown \tilde{C}) \geq 2\pi + 2||x-y||,$$
hence
$$\len(C) \geq 2\pi + 2||x-y|| - \len(\tilde{C}) \geq 2\pi + 2||x-y|| - \varphi(||a-b||).$$
\end{proof}
This proposition can easily be generalized.

\begin{prop}\label{proposition: length of closed curves, multi-curbe version}
Let $C_i:[0,1]\to \R^3$, $i=1,\dots,k$ be curves. Let $C$ be the curve obtained by interpolating between curves $C_i$ and $C_{i+1}$ with straight line segments $s_{i,i+1}$ (here $k+1$ is interpreted as $1$). Let $m$ be in the convex hull of $C$, $K$ the $m$-cone of $C$, and let $x,y \in K$ with $1\leq \dist(x,C_i),\dist(y,C_i)$ for $i=1,\dots,k$ and $||x-y||,||x-m||,||y-m||<\tfrac{1}{2}$. Let $\varphi(t) = \arcsin(2t)-t$. Then 
$$\sum_{i=1}^k \len(C_i) \geq 2\pi + 2||x-y|| - \sum_{i=1}^k \varphi\left(||C_i(1)-C_{i+1}(0)||\right).$$

\end{prop}
The following theorem is an \textit{almost-closed} analog of \propref{proposition: main inequality without epsilons} and an easy consequence of the previous proposition.

\begin{thm}\label{theorem: estimate of length of almost closed curves for multiple curves}
Let $C_i:[0,1]\to \R^3$, $i=1,\dots,k$ be curves. Let $C$ be the curve obtained by interpolating between curves $C_i$ and $C_{i+1}$ with straight line segments $s_{i,i+1}$ (here $k+1$ is interpreted as $1$). Let $m$ be in the convex hull of $C$, $K$ the $m$-cone of $C$, and let $A\subseteq K$ with $1\leq \dist(A,C_i)$ for $i=1,\dots,k$ and $\sup\{||x-m|| \mid x\in A\},\diam(A)<\tfrac{1}{2}$. Let $\varphi(t)=\arcsin(2t)-t$. Then
$$\sum_{i=1}^k \len(C_i) \geq 2\pi + 2\diam(A) -  \sum_{i=1}^k \varphi\left(||C_i(1)-C_{i+1}(0)||\right).$$
\end{thm}

Now that we have proved the main tool for our upcoming construction, we will state some technical lemmas. Since their proofs are standard, we have omitted them. 
\begin{lem}\label{lemma: stability of transverse intersections}
Let $D$ be a parametrized disk, and let $C$ be a curve such that $C$ and $D$ intersect transversely. Then there is an $\varepsilon>0$ such that for every parametrized disk $\tilde{D}$ with the same domain as $D$ and every curve $\tilde{C}$ with the same domain as $C$,
 $$||D-\tilde{D}||, ||C-\tilde{C}|| < \varepsilon \implies \tilde{C} \cap \tilde{D} \neq \emptyset.$$
\end{lem}

\begin{lem}\label{lemma: intersection of compact and closed set}
Let $A \subseteq \R^3$ be compact and $B\subseteq \R^3$ be closed with $A \cap B = \{p\}$. Then, for every $\varepsilon>0$ there exists a $\delta>0$ such that for all $x\in A$, $\dist(x,B)<\delta \implies \dist(x,p)<\varepsilon$.
\end{lem}

\begin{lem}\label{lemma: transverse intersections stay close when perturbing}
Let $D$ be a parametrized disk and let $C$ be a curve, both with compact domains, such that $C$ intersects $D$ in exactly one point $p\in \R^3$ and such that this intersection is transverse. For every $\varepsilon>0$ there exists $\delta>0$ such that for every parametrized disk $\tilde{D}$ with the same domain as $D$, every curve $\tilde{C}$ with the same domain as $C$, and every $\tilde{q}\in\tilde{C}\cap\tilde{D}$,
$$||D-\tilde{D}||, ||C-\tilde{C}||<\delta \implies ||\tilde{q}-p||<\varepsilon.$$
\end{lem}

\begin{lem}\label{lemma: breaking point clustering is nice}
Let $C_1,C_2,\dots,C_n$ be $n$ curves, all defined on compact intervals such that $p=C_1(x_1)=C_2(x_2)=\dots=C_n(x_n)$ is the only intersection point of $C_1$ and $C_2$. Further assume that every $C_i$ is an embedding. For every $\varepsilon>0$ there exists $\delta>0$ such that the following holds: Let $||C_i - \tilde{C}_i||<\delta$ for all $i$. Let $y_1,\dots,y_n$ be such that $$(\tilde{q}_1,\tilde{q}_2,\dots,\tilde{q}_n) = (\tilde{C}_1(y_1),\tilde{C}_2(y_2),\dots,\tilde{C}_n(y_n)) \in \tilde{C}_1 \times \tilde{C}_2 \times \dots \times \tilde{C}_n$$ minimizes $\diam\left(\{v_1,v_2,\dots,v_n\}\right)|_{v_i\in \tilde{C}_i}$. Then for all $i,j \in \{1,\dots, n\}$, we have
$$||\tilde{q}_i-\tilde{q}_j||, ||\tilde{q}_i-p||,|x_i - y_i|<\varepsilon.$$
\end{lem}

\begin{lem}\label{lemma: subcurve stays close when perturbing}
Let $C: I \to \R^3$ be a uniformly continuous curve and $\varepsilon>0$. Then there exists $\delta>0$ such that for all $C'$ with $||C'-C||<\delta$ and all $a,b,a',b'\in I$ with $a<b$, $a'<b'$ and $|a-a'|,|b-b'|<\delta$, we have $||\bar{C}-\bar{C}'||<\varepsilon,$ where
$$\bar{C}, \bar{C}': [\min(a,a'),\max(b,b')] \to \R^3$$
$$\bar{C}(t) = \begin{cases} C(t) \text{ if } t\in [a,b]\\ C(a) \text{ if } t<a \\ C(b) \text{ if } t>b\end{cases}, \quad \bar{C}'(t) = \begin{cases} C'(t) \text{ if } t\in [a',b']\\ C'(a') \text{ if } t<a' \\ C'(b') \text{ if } t>b'\end{cases}.$$
\end{lem}

\begin{lem}\label{lemma: interpolation of subcurves stays close when perturbing}
Let $C:I \to \R^3$ be a uniformly continuous curve and let $\varepsilon>0$. Then there exists $\delta>0$ such that the following holds. Let $C'$ satisfy $||C'-C||<\delta$, let $C_i:[a_i,b_i]\to\R^3$, $i=1,\dots,n$, be subcurves of $C$, and let $C'_i:[a'_i,b'_i]\to\R^3$ be subcurves of $C'$ such that
$$|a_i-a'_i|,|b_i-b'_i|<\delta.$$
Then
$$||\widehat{C}-\widehat{C}'||<\varepsilon,$$
where
$$\widehat{C}=\bar{C}_1\frown S_1\frown\bar{C}_2\frown S_2\frown\dots\frown\bar{C}_n\frown S_n,$$
$$\widehat{C}'=\bar{C}'_1\frown S'_1\frown\bar{C}'_2\frown S'_2\frown\dots\frown\bar{C}'_n\frown S'_n$$
are closed curves defined as follows: $\bar{C}_i$ and $\bar{C}'_i$ are defined as in \lemref{lemma: subcurve stays close when perturbing}, and $S_i,S'_i$ are straight line segments interpolating between the curves $\bar{C}_i,\bar{C}_{i+1}$ and the curves $\bar{C}'_i,\bar{C}'_{i+1}$ respectively (here indices are understood to be cyclic).
\end{lem}

\begin{thm}
Let $L=C_1 \cup C_2$ be a two-component link, with a suitable conjoined circle decomposition. Then there is a neighborhood $U$ around $L$, such that
$$L'\in U \implies \len(L')\geq \len(L)$$
with equality if and only if $L'$ has a suitable conjoined circle-decomposition with the same index sets, the same index functions $h^A$ and $h^B$, and the same linking scheme. 
\end{thm}

\begin{proof}
Let $\tilde{L}=\tilde{C}_1 \cup \tilde{C}_2$ be a thick link with  $||L-\tilde{L}||=d$. 
\\
For each $A$-index $l$ and $B$-index $l$, pick closed intervals $\{H_i\}_{i\in \alpha_l}$ with $a_i \in H_i^{\circ}$, and pick closed intervals $\{H_i\}_{i\in \beta_l}$ with $b_i \in H_i^{\circ}$, such that $\{H_i \mid i \text{ }A\text{-index}\}$ are pairwise disjoint and such that $\{H_i \mid i \text{ }B\text{-index}\}$ are pairwise disjoint. Now, define sets of points $\{\tilde{a}_i\}_{i \in \alpha_l}$ such that $\tilde{a}_i \in H_i$ and $\diam(\tilde{C}_1(\{\tilde{a}_i\}_{i \in \alpha_l}))$ is minimized, and $\{\tilde{b}_i\}_{i \in \beta_l}$  with $\tilde{b}_i \in H_i$ and such that $\diam(\tilde{C}_2(\{\tilde{b}_i\}_{i \in \beta_l}))$ is minimized. Since the circle decomposition is suitable, we can assume that there are always two curves in $\{C_1 \mid_{H_i}\}_{i \in \alpha_l}$ only intersecting in $C_1(\alpha_l)$ and two curves in $\{C_2 \mid _{H_i}\}_{i \in \beta_l}$ only intersecting in $C_2(\beta_l)$. This allows us to apply \lemref{lemma: breaking point clustering is nice} to the families of curves $\{C_1 \mid_{H_i}\}_{i \in \alpha_l}$ and $\{C_2 \mid _{H_i}\}_{i \in \beta_l}$. We get the following result.
\textbf{For any $\mathbf{\delta_1>0}$, there is $\mathbf{\delta>0}$ such that $\mathbf{d<\delta}$ implies}
$$|a_i-\tilde{a}_i|<\delta_1, \quad|b_i-\tilde{b}_i|<\delta_1,$$
$$\dist(C_1(\alpha_l),\tilde{C}_1(\tilde{a_i}))<\delta_1, \quad\dist(C_2(\beta_l),\tilde{C}_2(\tilde{b_i}))<\delta_1,$$
$$\diam(\tilde{C}_1(\alpha_l))<\delta_1, \quad \diam(\tilde{C}_2(\beta_l))<\delta_1,$$
where $\tilde{C}_1\left(\alpha_l\right) := \{\tilde{C}_1(\tilde{a}_i) \mid i\in \alpha_l\}$ and $\tilde{C}_2(\beta_l):=\{\tilde{C}_2(\tilde{b}_i) \mid i\in \beta_l\}$. For $A$-indices $i$, we define $\tilde{I}_i:=[\tilde{a}_{i},\tilde{a}_{i+1}]$ and for $B$-indices $i$ we define $\tilde{I}_i:=[\tilde{b}_i,\tilde{b}_{i+1}]$. The decompositions $\left(\tilde{C}_1,\{\tilde{a}_i\},\{A_j\},\{\alpha_l\}\right)$ and $\left(\tilde{C}_2,\{\tilde{b}_i\},\{B_j\},\{\beta_l\}\right)$ should be thought of as almost-closed versions of a conjoined circle decomposition of $\tilde{L}$. To every circle $K_j$ of the link $L$ we associate a spanning disk $D_j$ and a center $z_j$. To investigate the length of $\tilde{L}$, we define perturbed versions of $z_j$ and $D_j$ for $\tilde{C}_1$, the case for $\tilde{C}_2$ follows analogously. \lemref{lemma: interpolation of subcurves stays close when perturbing} gives us that, \textbf{for every $\mathbf{\delta_2>0}$} the straight line interpolation $\tilde{C}_1^j$ between the curves $\{\tilde{C}_1 |_{\tilde{I}_i} \mid i \in A_j\}$ satisfies $$||C_1^j-\tilde{C}_1^j||<\delta_2$$\textbf{ as long as $\mathbf{\delta}$ and $\mathbf{\delta_1}$ are sufficiently small.} We define $\tilde{z}_j$ as the center of mass of $\tilde{C}_1^j$ and we define $\tilde{D}_j$ as the $\tilde{z}_j$-cone of $\tilde{C}_1^j$. Since the center of mass of a curve varies continuously with that curve, $||\tilde{z}_j-z_j||$ can be made arbitrarily small. It follows that, as parametrized disks, $||\tilde{D}_j-D_j||$ can be made arbitrarily small. \textbf{Let $\delta_3>0$. For $\delta$, $\delta_1$, $\delta_2$ small enough we have
$$||\tilde{C}_1^j - C_1^j||,||\tilde{D}_j-D_j||<\delta_3.$$} Now let $j$ be an $A$-index and let $l=h^B(j)$, so $C_2(\beta_l)=z_j$. Let $i \in \beta_l$. Then $C_2(b_i) = z_j \in D_j$, and this intersection is transverse, so $C_2|_{H_i}$ and $D_j$ have a transverse intersection. By \lemref{lemma: stability of transverse intersections} and \lemref{lemma: transverse intersections stay close when perturbing}, \textbf{for sufficiently small $\delta_3$, we have} $S_{i,j} \coloneqq \tilde{C}_2(H_i) \cap \tilde{D}_j \neq \emptyset$ and that for $S_j\coloneqq \bigcup\limits_{i \in \beta_l} S_{i,j}$
$$\sup\{||x-\tilde{z}_j|| \mid x\in S_j\},\diam(S_j)<\frac{1}{2}.$$
By construction of $\{\tilde{b}_i \mid i\in \beta_l\}$ we have $\diam(\tilde{C}_2(\beta_l)) \leq \diam(S_j).$
Take $\tilde{z}_j$ as the point in the convex hull of $\tilde{C}_1^j$. Note that by construction $S_j$ is a subset of $\tilde{D}_j$, which is the $\tilde{z}_j$-cone of $\tilde{C}_1^j$, and note that $\sup\{||x-\tilde{z}_j|| \mid x\in S_j\}<\tfrac{1}{2}$.  Let $\varphi(x) = \arcsin(2x)-x$ and consider $C_1^j  = C_{i_1}^{s_1}\frown C_{i_2}^{s_2}\frown \dots\frown C_{i_c}^{s_c}.$
For $s_k=1$, let $x_{k}=i_k$ and $y_k=i_k +1$. For $s_k=-1$, let $x_k=i_k+1$ and $y_k=i_k$. Then $C_{i_k}^{s_k}$ has startpoint $C_1(a_{x_k})$ and endpoint $C_1(a_{y_k})$. By \theoref{theorem: estimate of length of almost closed curves for multiple curves}, we have
\begin{align*}
\sum\limits_{i \in A_j}\len\left(\tilde{C}_1 |_{\tilde{I}_i}\right) \geq &2\pi + 2\diam(S_j) - \sum_{k =1}^c
\varphi(||\tilde{C}_1(\tilde{a}_{y_k})-\tilde{C}_1(\tilde{a}_{x_{k+1}})||)\\
\geq &2\pi + 2\diam(\tilde{C}_2(\beta_{h^B(j)})) - \sum_{k=1}^c
\varphi(\diam(\tilde{C}_1(\alpha_{l(y_k)})))\\
=& 2\pi + 2\diam(\tilde{C}_2(\beta_{h^B(j)})) - \smashoperator{\sum_{\alpha_l \text{ along } K_j}} \varphi(\diam(\tilde{C}_1(\alpha_{l})))
\end{align*}
where we used that $x_{k+1},y_k\in \alpha_{l(y_k)}$. Let $q^A(l)= \diam(\tilde{C}_1(\alpha_l))$, $q^B(l)= \diam(\tilde{C}_2(\beta_l))$, and $f^A(l)=\#\{j \mid \alpha_l \text{ along } C_1^j\}$,  $f^B(l)=\#\{j \mid \beta_l \text{ along } C_2^j\}$. Let $n=\#\{A\text{-indices }j\}$. By summing over all $j$, we get

$$\len(\tilde{C}_1) \geq2n\pi + 2\sum\limits_{j \text{ } A\text{-index }} q^B(h^B(j))- \sum_{l \text{ } A\text{-index}} f^A(l) \cdot \varphi(q^A(l)).$$

Exactly analogously, for $m=\#\{B\text{-indices } j\}$, we have
$$\len(\tilde{C}_2) \geq2m\pi + 2\sum\limits_{j \text{ } B\text{-index }} q^A(h^A(j))- \sum_{l \text{ } B\text{-index}} f^B(l) \cdot \varphi(q^B(l)).$$
In total, we get
\begin{align*}
\len(\tilde{C}_1)+\len(\tilde{C}_2) \geq &2(n+m)\pi  \\
+&\sum_{l \text{ }A\text{-index}} 2\#(h^A)^{-1}(l) \cdot q^A(l) -f^A(l)\cdot \varphi(q^A(l)) \\
+&\sum_{l \text{ }B\text{-index}} 2\#(h^B)^{-1}(l) \cdot q^B(l) - f^B(l) \cdot \varphi(q^B(l))
\end{align*}
Let $$\mu = \min\left(\left\{\frac{2\#(h^A)^{-1}(l)}{f^A(l)}\right\}\cup\left\{\frac{2\#(h^B)^{-1}(l)}{f^B(l)}\right\}\right)>1.$$ Since near zero we have that $\varphi(x)\sim x$ to first order, \textbf{there exists $\delta_4>0$} such that on $[0,\delta_4]$ we have $\varphi(x)\leq \mu x$ with equality if and only if $x=0$. In particular, since $q^A(l),q^B(l)<\delta_1$, \textbf{for $\delta_1<\delta_4$, we have}
$$2\#(h^A)^{-1}(l)\cdot q^A(l) -  f^A(l) \cdot \varphi(q^A(l)) \geq 0,$$
$$2\#(h^B)^{-1}(l)\cdot q^B(l) - f^B(l)\cdot \varphi(q^B(l)) \geq 0,$$
with equality if and only if $q^A(l)=0$ and $q^B(l)=0$ respectively. Thus, for $d=||\tilde{L}-L||$ small enough, we have
$$\len(\tilde{L}) = \len(\tilde{C}_1)+\len(\tilde{C}_2) \geq 2(n+m)\pi = \len(C_1)+\len(C_2) = \len(L).$$
Therefore $L= C_1 \cup C_2$ is a local minimum. Assume now that $\len(\tilde{L})=\len(L)$.\\
\textbf{Claim: Both $\tilde{C}_1$ and $\tilde{C}_2$ circle-decompose with the same index sets as $C_1$ and $C_2$ respectively, and the linking scheme stays the same. Further, the circle decomposition of $\tilde{C}_1$ and $\tilde{C}_2$ is conjoined and suitable.} We have the equality $\len(\tilde{C}_1)+\len(\tilde{C}_2) = 2(n+m)\pi$ if and only if $q^A(l)=0$ for every $A$-index $l$ and $q^B(l)=0$ for every $B$-index $l$. This implies that all clusters $\alpha_l$ and $\beta_l$ map to a single point. So the straight line interpolations used to define the curves $\tilde{C}_1^j$ and $\tilde{C}_2^j$ were unnecessary. Since every circle $K_j$ of one of the curves was linked with some circle $K_{j'}$ of the other curve, and since $\tilde{L}$ is close to $L$, every curve $\tilde{C}_1^j$ will be linked with some curve $\tilde{C}_2^{j'}$ and vice-versa. Every curve $\tilde{C}_1^j$ and every curve $\tilde{C}_2^j$ needs to have length at least $2\pi$, and through the global length constraint of $\len(\tilde{L})=2(n+m)\pi$ we get that every curve $\tilde{C}_1^j$ and every curve $\tilde{C}_2^j$ has to have length exactly $2\pi$, hence is a unit circle by \theoref{theorem: Gehring Problem}. This shows that $\tilde{C}_1$ and $\tilde{C}_2$ have circle decompositions with the same index sets as $C_1$ and $C_2$ respectively. That the linking scheme is the same is trivial since for every pair of indices $(j,j')$, the link $\tilde{K}_j \cup \tilde{K}_{j'}$ is close to $K_j \cup K_{j'}$. It is easy to check that the circle decomposition of $\tilde{C}_1$ and $\tilde{C}_2$ is conjoined, with the same functions $h^A$ and $h^B$ and the same linking scheme as the conjoined circle decomposition of $C_1$ and $C_2$, and that this circle decomposition is therefore suitable.
\end{proof}

We will now prove \theoref{theorem: main theorem of chapter 4}.

\begin{proof}[Proof of \theoref{theorem: main theorem of chapter 4}]
In the last theorem we showed that $S$ is strictly locally minimizing. It remains to be shown that $S$ is closed. Take a sequence $(L_n)_{n\in \N}$ in $S$ converging to some thick link $L=(C_1,C_2)$. Let $\{A_j\},\{\alpha_l\},\{B_j\},\{\beta_l\}$ be the index sets and let $h^A, h^B$ be the index functions belonging to $S$. For each $n\in \N$, consider the sets of breaking points $\{a_{i,n}\}, \{b_{i,n}\} \subseteq S^1$. Since $S^1$ is compact, we can assume that all these sequences $(a_{i,n})_{n\in \N}$ and $(b_{i,n})_{n\in \N}$ converge, else take a subsequence of $(L_n)_{n\in \N}$. We denote the limits by $a_i$ and $b_i$ respectively. Note that for each fixed $n$, the cyclic ordering of $\{a_{i,n}\}$ as well as that of $\{b_{i,n}\}$ are the same, hence the cyclic orderings of $\{a_i\}$ and $\{b_i\}$ are (weakly) preserved when taking limits (but a priori it seems possible that $a_i=a_{i+1}$ or $b_i=b_{i+1}$ for some $i$). We will now show that $(\{A_j\},\{\alpha_l\},\{a_i\}),(\{B_j\},\{b_i\},\{\beta_l\})$ form a suitable conjoined circle decomposition of $L$.\
\\
\textbf{Claim: $\{A_j\}$, $\{\alpha_l\}$, $\{a_i\}$ form a circle decomposition of $C_1$ and $\{B_j\}$, $\{b_i\}$, $\{\beta_l\}$ form a circle decomposition of $C_2$.} We show this claim for $C_1$, the proof for $C_2$ is the same. It is obvious that circles $C_{1,n}^j$ of $C_{1,n}$ converge to circles $C_1^j$ of $C_1$. It is left to show that
$$C_1(a_{i_1})=C_1(a_{i_2}) \iff \text{ there is some } l \text{ with } i_1,i_2 \in \alpha_l.$$
It is easy to see that we have $i_1,i_2 \in \alpha_l \implies C_1(a_{i_1})=C_1(a_{i_2}).$ Now assume that $i_1 \in \alpha_{l_1}$ and $i_2 \in \alpha_{l_2}$ with $l_1\neq l_2$. Let $K_{j_1}$ and $K_{j_2}$ be the unit circles of $C_2$ with $z_{j_1}=C_1(\alpha_{l_1})$ and $z_{j_2}=C_1(\alpha_{l_2})$. Since $l_1\neq l_2$, we have that $C_{1,n}(\alpha_{l_1})\neq C_{1,n}(\alpha_{l_2})$ for every $n\in \N$. It follows that for every $n\in \N$ we have $z_{j_1,n} \neq z_{j_2,n}$ and therefore $K_{j_1,n} \neq K_{j_2,n}$. Since the circle decompositions of every link $L_n$ are suitable, there is some $A$-index $j'$ such that for every $n$ the circle $K_{j',n}$ is non-trivially linked with $K_{j_1,n}$, but trivially linked with $K_{j_2,n}$. Taking limits, we get that $K_{j'}$ is non-trivially linked with $K_{j_1}$ but trivially linked with $K_{j_2}$, hence $K_{j_1}\neq K_{j_2}$. This implies $z_{j_1}\neq z_{j_2}$, since if two different circles $K_{j_1}$, $K_{j_2}$ of $C_2$ had the same center, it would be impossible for $C_1$ to pierce their common center while keeping unit distance from both. Therefore $C_1(a_{i_1})=z_{j_1}\neq z_{j_2}=C_1(a_{i_2})$. This shows that two $A$-indices $i_1,i_2$ are in the same point cluster if and only if $C_1(a_{i_1})=C_1(a_{i_2})$. It is left to check that the intervals $I_i=[a_i,a_{i+1}]$ all have positive length, i.e.\ $a_i\neq a_{i+1}$. If $i,i+1$ are in different point clusters, then by what we have just shown $C_1(a_i)\neq C_1(a_{i+1})$ and hence $a_i\neq a_{i+1}$. If $i,i+1$ are in the same $\alpha_l$, then there is some $j$ with $A_j=\{i\}$. This means that $C_{1,n}([a_{i,n},a_{i+1,n}])=K_{j,n}$ is a unit circle for each $n$. By taking limits, $C_1([a_i,a_{i+1}])=K_j$ is a unit circle, hence $I_i$ has positive length.\\
\textbf{Claim: The circle decompositions of $C_1$ and  $C_2$ are conjoined.} We check the following five conditions:
\begin{itemize}
\item $\dist(C_1,C_2)=1$, 
\item $K_j$ circle of $C_1$ $\implies z_j \in C_2(\{b_i\})$ and every intersection of $C_2$ and $D_j$ is transverse
\item $K_j$ circle of $C_2$ $\implies z_j \in C_1(\{a_i\})$ and every intersection of $C_1$ and $D_j$ is transverse,
\item for every $\alpha_l$ there is some circle $K_j$ of $C_2$ such that $C_1(\alpha_l)=z_j,$
\item for every $\beta_l$ there is some circle $K_j$ of $C_1$ such that $C_2(\beta_l)=z_j.$
\end{itemize}
The first condition is obviously fulfilled. We will show that the second condition is fulfilled; the third one follows analogously. Let $K_j$ be a circle of $C_1$. Let $i$ be such that $z_{j,n} = C_{2,n}(b_{i,n})$ for every $n$. Then by taking limits, $z_j = C_2(b_i)$. That the intersection is transverse can be seen by observing that the arclength parametrization of $C_2$ is $\mathcal{C}^1$ and piecewise $\mathcal{C}^2$.  This fact, together with the observation that $C_2$ keeps unit distance from the boundary of the unit disk $D_j$, means that any intersection of the arclength parametrization of $C_2$ with the unit disk $D_j$ has to be transverse.\\
For every $A$-index $l$, suitability gives some $B$-index $j$ with $h^A(j)=l$. Thus $z_{j,n}=C_{1,n}(\alpha_l)$ for every $n\in\N$, and taking limits gives $z_j=C_1(\alpha_l)$. This proves the fourth condition. The fifth condition follows analogously.\\
The conjoined circle decomposition has the same index functions $h^A$ and $h^B$ as that of $L$. Note that taking limits preserves both thick Hopf links and thick unlinks, so we have $K_j \text{ linked with } K_{j'} \iff K_{j,n} \text{ linked with } K_{j',n}.$ This shows that $L$ has the same linking scheme as each $L_n$.\\
\textbf{Claim: The conjoined circle decomposition of $L$ is suitable.} The conditions for suitability are
\begin{itemize}
\item  Every circle $K_j$ of one of the curves is linked with exactly two circles $K_{j_1},K_{j_2}$ of the other curve,
\item for every $A$-index  $l$: $\#\{j \text{ } A\text{-index } \mid \alpha_l \text{ along } K_j\}<2\#(h^A)^{-1}(l)$,
\item for every $B$-index  $l$: $\#\{j \text{ } B\text{-index } \mid \beta_l \text{ along } K_j\}<2\#(h^B)^{-1}(l)$,
\item for every $\alpha_l$, there are two distinguishable $i_1,i_2\in \alpha_l$,
\item for every $\beta_l$, there are two distinguishable $i_1,i_2\in \beta_l$.
\end{itemize}
The first condition follows trivially from the linking scheme of $L$ and $L_n$ being the same. For the second condition, note that
$$\{j \text{ } A\text{-index } \mid \alpha_l \text{ along } C_1^j\} =\{j \text{ } A\text{-index } \mid \alpha_l \text{ along } C_{1,n}^j\}.$$ Since $h^A, h^B$ are the index functions of both $L$ and $S$, the second condition is fulfilled. The third condition holds analogously. It remains to verify the last two conditions. We show this for the point clusters $\alpha_l$; the proof for the point clusters $\beta_l$ follows analogously. Since there are only finitely many point clusters and each point cluster is finite, after passing to a subsequence we may choose for every $\alpha_l$ two fixed indices $i_1,i_2\in\alpha_l$ which are distinguishable for every $C_{1,n}$.

Fix $\alpha_l$ and the corresponding indices $i_1,i_2$. The two one-sided circle arcs at $a_{i,n}$ are the germs at $C_{1,n}(\alpha_l)$ of the restrictions of $C_{1,n}$ to the two intervals adjacent to $a_{i,n}$. Suppose that $i_1$ and $i_2$ were not distinguishable in the limit. Then one of the one-sided circle arcs at $a_{i_1}$ would agree with one of the one-sided circle arcs at $a_{i_2}$. Let $K_{j_1,n}$ and $K_{j_2,n}$ be the circles supporting the corresponding one-sided arcs at $a_{i_1,n}$ and $a_{i_2,n}$. Their limiting circles $K_{j_1}$ and $K_{j_2}$ contain a common non-trivial circle arc and are therefore the same geometric circle.

Since $K_{j_1}=K_{j_2}$, these circles have the same linking partners. We have already observed that taking limits preserves the linking scheme, so $K_{j_1,n}$ and $K_{j_2,n}$ have the same linking partners for every $n$. Since linking partners uniquely determine geometric circles, we have $K_{j_1,n}=K_{j_2,n}$ for every $n$. As $i_1$ and $i_2$ are distinguishable, the two selected one-sided arcs must therefore be the two opposite sides of this circle at their common point. Opposite sides remain opposite under taking limits, contradicting our assumption that the limiting one-sided arcs agree. Hence $i_1$ and $i_2$ remain distinguishable in the limit.
\end{proof}

We can now apply \propref{proposition: every point of a closed, strictly minimizing set is a sink}.

\begin{cor}
Every $2$-component thick link $L$ that has a suitable conjoined circle decomposition is a local minimum for ropelength and a sink.
\end{cor}

\subsection{A Gordian unlink and other examples}\label{subsection: A Gordian unlink and other examples}

\begin{rem}\label{remark: linking number is a complete link invariant}
When considering $2$-component links up to link homotopy, the linking number is a complete link invariant, see \cite{LinkGroups}. This means a $2$-component link is (link homotopically) an unlink if and only if its linking number is zero.
\end{rem}

\begin{thm}\label{theorem: a Gordian unlink exists}
A two-component homotopically Gordian unlink exists.
\end{thm}

\begin{proof}
One carefully checks that the link shown in \figref{figure: 2-component gordian unlink}, parametrized as indicated in \figref{figure: circle decomposition augmented} has a suitable conjoined circle decomposition, and that the link has linking number equal to zero.
\end{proof}

\begin{thm}\label{theorem: Gordian pairs exist in all link homotopy classes of two components}
For every link homotopy class with $2$ components, there exists a thick link in that class which is a non-global, local minimum and a sink for ropelength. In particular, there are Gordian pairs in every link homotopy class with $2$ components. 
\end{thm}

\begin{proof}
By \remref{remark: linking number is a complete link invariant}, it therefore suffices to show that for every $n\in\Z$, there exists a $2$-component thick link $L$ with linking number $n$ which is a non-global, local minimum and a sink. It suffices to show the claim for non-negative $n$, as the linking number of a $2$-component link changes by a factor of $-1$ when changing the orientation of one of the curves in it link.\\
Let $(C_1,C_2)$ be the Gordian unlink from \theoref{theorem: a Gordian unlink exists} and with parametrization indicated as in \figref{figure: circle decomposition augmented}. If $n=0$, we can just take $L=(C_1,C_2)$ and we are done, so let $n\geq 1$. Let $C_1$ be the curve with colors yellow, green, and blue, parametrized by starting at yellow-grey breaking point of the point cluster $F$ and going along the grey segment. Let $S_1$ be the curve parametrizing the grey unit circle around $E$ of the sketch, parametrized with the same orientation it has as a part of $C_1$. Let $C_2$ be the curve with the colors red, purple, and orange, parametrized by starting at the red–grey breaking point of the point cluster $A$ and going along the grey segment. Let $S_2$ be the curve parametrizing the grey unit circle around $F$, parametrized with the same orientation it has as a part of $C_2$. Now consider $\tilde{C}_1$ defined by concatenating $k$ copies of $C_1$ and $n$ copies of $S_1$, and consider $\tilde{C}_2$ defined by concatenating $k$ copies of $C_2$ and one copy of $S_2$. Set $L=(\tilde{C}_1,\tilde{C}_2)$.\\ 
Let $\link(\cdot,\cdot)$ refer to the linking number of two curves. Then 
\begin{align*}
    \link(\tilde{C}_1,\tilde{C}_2) =&k^2\cdot \link(C_1,C_2)+ nk\cdot\link(S_1,C_2)\\
    +&k\cdot\link(C_1,S_2) + n\cdot\link(S_1,S_2)\\
    =& 0+0+0+n\cdot\link(S_1,S_2) =   n
\end{align*}  
where we used that $\link(C_1,C_2)=\link(S_1,C_2)=\link(C_1,S_2)=0$. It is trivial to give $(\tilde{C}_1,\tilde{C}_2)$ a conjoined circle decomposition. As long as $k\geq \max(2,n-1)$, this circle decomposition will be suitable. This shows that $L$ is a local minimum and a sink. For $n\geq 1$, there is a thick link $L'$ in the link homotopy class of $L$ with $\len(L')=2(n+1)\pi$. On the other hand,
$$\len(L)=24k\pi+2(n+1)\pi=(24k+2n+2)\pi,$$
so $\len(L')<\len(L)$. Thus $L$ is not a global minimum and $(L,L')$ is a Gordian pair.
\end{proof}

\section*{Acknowledgements}
The author is grateful to John Sullivan for suggesting the study of thick links related to the $4$-component configuration described in \sectref{An isotopically Gordian split link exists}. The author thanks the anonymous referee for pointing out numerous small mistakes and for helpful suggestions which improved the manuscript. The author thanks Keeran Cross for creating \figref{figure: parametrization of circle decomposition} and \figref{figure: circle decomposition augmented} based on hand-drawn sketches and instructions provided by the author.

\end{document}